\documentclass[report]{shapeworkshop}

\usepackage{enumitem}
\usepackage{float}
\usepackage{fontawesome}
\usepackage{amssymb}
\usepackage{graphicx}
\usepackage[all]{xy}
\usepackage[hidelinks]{hyperref}

\reporter{Philipp Harms}  
\reporterprefix{Editor:}
\reportyear{2018} 
\reportnumber{} 
\reportname{Math in the Black Forest: Workshop on New Directions in Shape Analysis} 
\reportdate{16 July -- 22 July 2018} 
\organizers{Philipp Harms, University of Freiburg} 

\newcommand{\myhref}[2]{\href{#1}{#2\,{\footnotesize\faExternalLink}}}
\def\Diff{\operatorname{Diff}}
\def\Div{\operatorname{div}}
\def\vol{\operatorname{vol}}
\def\Jac{\operatorname{Jac}}
\newcommand{\ud}{\,\mathrm{d}}
\def\R{{\mathbb R}}
\def\S{{\mathbb S}}
\newtheorem{question}{Question}

\newtheorem{definition}{Definition}
\newtheorem{theorem}{Theorem}
\newtheorem*{theorem*}{Theorem}
\newtheorem{conjecture}{Conjecture}
\newtheorem{problem}{Problem}
\newtheorem{proposition}{Proposition}
\newcommand{\eqdef}{\ensuremath{\stackrel{\mbox{\upshape\tiny def.}}{=}}}

\begin{document}

\maketitle

\begin{introduction}
\noindent
The workshop brought together researchers in shape analysis to discuss promising new directions. 
Shape analysis is an inter-disciplinary area of research with theoretical foundations in infinite-dimensional Riemannian geometry, geometric statistics, and geometric stochastics, and with applications in medical imaging, evolutionary development, and fluid dynamics.
It is the 6th instance of a series of workshops, which were held at the following locations:
\begin{enumerate}
\item[6.] \textbf{Rinkenklause am Feldberg, Germany, 2018.} \\ Supported by University of Freiburg. \myhref{https://www.stochastik.uni-freiburg.de/professoren/philipp-harms/shapeworkshop}{Website}
\item[5.] \textbf{Skrunda, Latvia, 2017.} \\ Supported by Brunel University London. \myhref{http://www.brunel.ac.uk/~mastmmb/mathlatvia.html}{Website}
\item[4.] \textbf{Tende, France, 2016.} \\ Supported by INRIA.
\item[3.] \textbf{Grundsund, Sweden, 2015.} \\ Supported by Chalmers University. \myhref{https://www.chalmers.se/sv/institutioner/math/forskning/konferenser-pa-MV/Shape-Analysis-Meeting/Pages/default.aspx}{Website} \myhref{http://dx.doi.org/10.5281/zenodo.33558}{Proceedings}
\item[2.] \textbf{Bad Gastein, Austria, 2014.} \\ Supported by the Austrian Science Foundation. \myhref{http://www.brunel.ac.uk/~mastmmb/mathcabin.html}{Website} \myhref{https://hal.archives-ouvertes.fr/hal-01076953v1}{Proceedings}
\item[1.] \textbf{Foxton Beach, New Zealand, 2013.} \\ Supported by Massey University.
\end{enumerate}
\end{introduction}
The participants of the workshop are: 
Martin Bauer, 
Martins Bruveris, 
Nicolas Charon, 
Philipp Harms, 
Sarang Joshi, 
Boris Khesin, 
Alice Le Brigant,
Elodie Maignant, 
Stephen Marsland,
Peter Michor, 
Xavier Pennec,
Stephen Preston, 
Stefan Sommer, 
Fran\c{}ois-Xavier Vialard, 

\medskip\noindent
{\em Acknowledgement:} Philipp Harms would like to thank the University of Freiburg for supporting the workshop.

\clearpage

\tableofcontents

\begin{report}

\begin{talk}{Martin Bauer}
{Geodesic completeness of weak Riemannian metrics: A conjecture}
{Martin Bauer}

In this note we aim to introduce a conjecture on geodesic completeness of weak Riemannian metrics. 
We refer to \cite{BBM2014} to an introduction on infinte dimensional manifolds in a simialar notation.

\begin{conjecture}
Let $G^1$ and $G^2$  be smooth, weak Riemannian metrics on an infinte-dimensional Banach manifold $\mathcal H$ that satisfy:
\begin{enumerate}
\item The metric $G^2$ is uniformly stronger then the metric $G^1$, i.e., there exists a constant $C$ such that for all points $x\in\mathcal H$ and tangent vectors $h_x, k_x$ we have $G^1_x(h_x,k_x)\leq C G^2_x(h_x,k_x)$;
\item Both metrics admit a smooth geodesic spray and thus the Riemannian exponential map are  local diffeomorphisms
\item The metric $G_1$ is geodesically complete, i.e., for any initial conditions $(x,h_x)$ the geodesic initial value problem has solutions for all time $t$
\end{enumerate}
Then the metric $G^2$ is geodesically complete.

These are only the minimal assumptions that will be required for this conjecture to have change to be correct. Possible additional assumptions that might be needed include:
\begin{enumerate}[start=4]
\item The space $\mathcal H$ is in addition a topological group and the metrics $G^1$ and $G^2$ are right invariant
\item The derivatives of the metric $G^1$ are controlled in terms of derivatives of $G^2$.
\end{enumerate}
\end{conjecture} 
By the theorem of Hopf-Rinow this conjecture is trivially satisfied in finite dimensions. In infinte dimensions, however, the answer to this question is widely open.

The motivation for this conjecture comes from the study of right invariant Sobolev metrics
on groups of diffeomorphisms of finite Sobolev regularity: it has been shown that right invariant Sobolev metric of order $s\in \mathbb R$ are geodesically complete if $s\geq \frac32$ \cite{EK2014a,PW2018}, while there always exist solutions that blow up in finite time for 
$s=0$, $s=\frac12$ \cite{okamoto2008generalization,PW2018,BKP2016} and $s=1$ \cite{CHH1994,constantin1998wave}. As a consequence of the above conjecture we would obtain geodesic incompleteness (blowup of solutions) 
for all metrics of order less than $s\leq 1$, which would reduce the open cases to the interval $1<s<\frac32$. Analogous results would follow for diffeomorphism groups of higher dimensional manifolds and reparametrization invariant metrics on more general spaces of mappings, e.g. 
on shape spaces of surfaces or on the space of all Riemannian metrics.

\def\cprime{$'$}

\end{talk}

\begin{talk}{Nicolas Charon}
{Riemannian metrics on shapes induced by measure embeddings}
{Charon, Nicolas}

\section{Measure representation of shapes}
Consider a space $\text{Emb}(M,\R^n))$ of smooth, parametrized and oriented submanifolds in $\R^n$ where $M$ is the parameter manifold of dimension $1\leq d \leq n$. The core principle of geometric measure theory frameworks such as currents, varifolds, normal cycles... is to view any $q \in \text{Emb}(M,\R^n)$ as a measure $\mu_{q}$ over a certain set, in such a way that $\mu_q$ is invariant to reparametrization. Here, we extend the approach of oriented varifolds introduced in \cite{Charon1}. We assume that a certain Banach space of test functions $W$ is given, with $W \hookrightarrow C_0(\R^n \times G_d^+)$ and $G_d^+$ the oriented Grassmannian of all oriented $d$-planes in $\R^n$. 

\begin{definition}
For any $q \in \text{Emb}(M,\R^n)$, we define the oriented varifold associated to $q$ as the measure $\mu_q \in W^*$ given by:
\begin{equation}
\mu_q(\omega) \doteq \int_M \omega\left(q(m),T_m q\right) d \vol(m)
\end{equation}
for all test function $\omega \in W$, $T_m q$ denoting the element of $G_d^+$ that represents the tangent space to $q$ at $m$.
\end{definition}
The change of variable formula allows to show that for any positive repara\-metrization $\phi \in \text{Diff}^+(M)$, we have $\mu_{q\circ \phi} = \mu_{q}$. In other words, the previous representation yields a quotient mapping $[\mu]$ from $\text{Emb}(M,\R^n))/\text{Diff}^+(M)$, the space of unparametrized oriented shapes, into the dual space $W^*$.

The mapping $[\mu]$ depends on the choice of $W$ and is not necessarily injective, in particular if the test function space is not 'rich enough'. It is injective when for example $W$ is dense in $C_0(\R^n \times G_d^+)$ although weaker conditions on $W$ can still lead to the injectivity of $[\mu]$ as well, cf \cite{Charon1}. In those cases, the dual metric on $W^*$ restricted to varifolds $\mu_{q}$ leads to a 'chordal' distance between shapes: $d^{chor}(q_1,q_2) = \|\mu_{q_1} - \mu_{q_2}\|_{W^*}$. So far, these types of distances have been regularly used as fidelity terms in diffeomorphic registration frameworks due to the fairly simple closed form expressions that one can derive when $W$ is for instance a Reproducing Kernel Hilbert Space. 

\section{The induced Riemannian metrics}
The previous chordal distances, however, have no direct interpretation on the shape space of submanifolds as the straight path in $W^*$ between $\mu_{q_1}$ and $\mu_{q_2}$ leaves, in general, the image of $\text{Emb}(M,\R^n)/\text{Diff}^+(M)$ by $[\mu]$. An alternative and possibly interesting idea, from a theoretical point of view, would be to consider instead the Riemannian (or Finsler) metrics induced by the embedding $[\mu]$ on $\text{Emb}(M,\R^n)/\text{Diff}^+(M)$.

The construction can be done as follows. Assuming enough regularity on the test functions in $W$ (typically $W \hookrightarrow C_0^2(\R^n \times G_d^+)$ suffices), one can express the tangent map to $\mu$ at any $q \in \text{Emb}(M,\R^n)$ as a mapping $T_{q} \mu: \ C^{\infty}(M,\R^n) \rightarrow W^*$. This is usually known as the first variation of the varifold and has been derived in a similar setting in 
\cite{Charon2}. Then, for $q \in \text{Emb}(M,\R^n)$ and a vector field $h \in C^{\infty}(M,\R^n)$, the induced metric writes:
\begin{equation}
 \|h\|_{q} \doteq \|T_q \mu(h) \|_{W^*}
\end{equation}
It is by construction a parametrization-invariant metric on $\text{Emb}(M,\R^n)$ for which we can expect the resulting distance on the quotient space $\text{Emb}(M,\R^n)/\text{Diff}^+(M)$ to be non-degenerate whenever $W$ is chosen so that $[\mu]$ is injective. Furthermore, some additional invariances like translation and rotation invariance can be easily recovered by carefully selecting the space $W$ and its metric (cf next section).
\vskip2ex
\noindent \textbf{Open questions}: this looks like a potentially quite modular approach to construct invariant metrics on spaces of curves, surfaces... but very little is clear at present about those metrics' properties or how they could relate to previously studied higher-order Sobolev or quasi-local metrics. It is also likely that the adequate space to consider from the start is $\text{Imm}(M,\R^n)$ instead of $\text{Emb}(M,\R^n)$, in which case the injectivity of $[\mu]$ cannot hold in all generality (some transversality conditions on the shapes are needed) and thus the non-degeneracy of the metric is not immediate.

\section{Case of closed curves}
In order to provide a slightly more explicit view of the above, let us specify to closed curves in $\R^n$ by taking $M=\S^1$. In this case, the oriented Grassmannian $G_d^+$ is simply the sphere $\S^{n-1}$ of all possible unit tangent vectors to a curve in $\R^n$. For a parametrized curve $c \in \text{Emb}(\S^1,\R^n)$, its associated varifold becomes:
\begin{equation}
\label{eq:var_curve}
\mu_c(\omega) = \int_{\S^1} \omega(c(\theta),\vec{t}_c(\theta)) ds
\end{equation}
where $ds = |c'(\theta)| d\theta$ and $\vec{t}_c(\theta) = c'(\theta)/|c'(\theta)| \in \S^{n-1}$. Its first variation can be here computed directly from \eqref{eq:var_curve} and after a few integration by parts, on finds:
\begin{equation}
\label{eq:tanmap_var_curve}
T_c\mu(h)(\omega) = \int_{\S^1} [\partial_x \omega - \partial^2_{x,u} \omega \cdot \vec{t}_c - \partial^2_{u,u} \cdot H_c -(\omega - \partial_u \omega \cdot \vec{t}) \cdot H_c]\cdot h^{\bot} ds
\end{equation}
in which $H_c$ is the curvature vector of $c$ and $h^{\bot}=(h\cdot \vec{n}) \vec{n}$ is the component of the vector field normal to $c$. 

Now, if we take $W$ as a Hilbert space and let $K_W: W^* \rightarrow W$ be the Riesz duality map, the induced Riemannian metric can be rewritten as 
\begin{equation*}
\|h\|_c^2 = (T_c\mu(h) | K_W T_c\mu(h))_{W^*,W}. 
\end{equation*}
As we also have by assumption the continuous embedding $W \hookrightarrow C_0(\R^n \times \S^{n-1})$, $W$ is a Reproducing Kernel Hilbert Space on $\R^n \times \S^{n-1}$ and $K_W T_c\mu(h)$ may be expressed based on the associated kernel function and its derivatives. 

In particular, it may be relevant to focus on the class of separable kernels considered in \cite{Charon1}, which are the product of a kernel of the form $\rho(|x-y|^2)$ on $\R^n$ and a kernel of the form $\gamma(u\cdot v)$ on $\S^{n-1}$. This choice indeed leads to a translation and rotation invariant metric on curves. The resulting expression for the metric involves multiple cross-terms with derivatives of $\rho$ and $\gamma$ of order up to 2:
\begin{align*}
\|h\|_c^2 &= \iint_{\S^1\times \S^1} \rho(|c(\theta)-c(\theta')|^2) \gamma(\vec{t}(\theta)\cdot \vec{t}(\theta')) (H_c(\theta)\cdot h^{\bot}(\theta)) (H_c(\theta') \cdot h^{\bot}(\theta')) ds ds' \\
&+\ldots
\end{align*}
with 8 additional terms of similar form. This general formula remains clearly more complex to interpret than Sobolev or quasi-local metrics. Yet studying some simpler particular cases, like $\gamma \equiv 1$ (standard measure metrics) or infinitely small scale for the kernel $\rho$ may provide clearer insight about the nature of those metrics.

\end{talk}

\clearpage

\begin{talk}{Philipp Harms and Elodie Maignant}
{Approximations of distances and kernels on shape space}
{Harms, Philipp}

\section{Motivation}

Shape spaces are inherently nonlinear: one cannot simply add or multiply two shapes and obtain a shape of the same type. 
Two successful and widely used machine learning algorithms which are able to deal with this nonlinearity are kernelized support vector machines (SVM) and multi-dimensional scaling (MDS).   
Both algorithms take as input a square matrix, whose entries are pairwise evaluations of a kernel in the case of SVMs and pairwise distances in the case of MDS. 
The calculation of these matrices can be prohibitively expensive for large datasets, and we suggest several approximations which are easier to compute. 

\section{Approximate distance matrices}
The calculation of pairwise distances between sample points on a Riemannian manifold can be prohibitively expensive if there are many samples or the dimension of the manifold is high. We propose two approximations which have the advantage that the number of geodesic boundary value problems to be solved is linear instead of quadratic in the number of samples. 
\subsection{Taylor expansion of the squared distance function}
It is well-known that the squared Riemannian distance function admits a Taylor expansion of the form
\begin{equation}\label{PH:taylor}\tag{$*$}
\operatorname{dist}^2_M(\exp_x(u),\exp_x(v))=\|u-v\|^2_{g(x)} - \frac13 R_x(u,v,v,u) + O(5),
\end{equation}
where $x$ is a point on the Riemannian manifold $(M,g)$, $u$ and $v$ belong to the tangent space at $x$, $R$ is the Riemannian curvature, and $O(5)$ denotes a function of $u,v$ which vanishes at zero of 5th order with respect to $g(x)$.
There are also higher-order expansions in the literature.
\subsection{Approximation by symmetric spaces}
The Taylor series \eqref{PH:taylor} has unrealistic polynomial fourth order growth for large $u$ and $v$. A remedy is to use the approximation\footnote{We are grateful to Peter Michor and Xavier Pennec for helping to make this idea precise.}
\begin{equation*}
\operatorname{dist}^2_M(\exp_x(u),\exp_x(v))\approx\operatorname{dist}^2_{\tilde M}(\exp_{\tilde x}(\tilde u),\exp_{\tilde x}(\tilde v)),
\end{equation*}
where the objects marked by a tilde are suitable approximations. For example, one may choose a two-dimensional symmetric space $(\tilde M,\tilde g)$ which matches the curvature at $x$ of the surface spanned by $u,v\in T_xM$. 
\section{Approximate heat kernels}
Heat kernels provide natural feature maps for kernel methods, but are expensive to compute. For sufficiently localized data they may be approximated using heat kernel expansions. Moderate deviations regimes seem better suited than the more commonly used large deviations regimes because they formalize a more intuitive notion of concentration of the data near a point and are easier to compute. It is an important challenge to develop fast (approximative) methods for evaluating heat kernels on Riemannian manifolds.
\end{talk}

\begin{talk}[Anton Izosimov]{Boris Khesin}
{Geometry behind vortex sheet}
{Khesin, Boris}

In 1966 V.I.~Arnold developed a group-theoretic approach to ideal hydrodynamics in which the Euler equation for an inviscid incompressible fluid is described as the geodesic flow equation for a right-invariant $L^2$-metric on the group of volume-preserving diffeomorphisms of the flow domain. 

This setting assumes sufficient smoothness of the initial velocity field $u$. 
In particular, it does not, generally speaking, describe flows with vortex sheets, i.e. with jump discontinuities 
in the velocity. On the other hand, it was recently discovered by F.\,Otto and C.\,Loeschcke 
that the motion of vortex sheets is also 
governed by a geodesic flow, but of somewhat different origin. 
Consider the space $VS(M)$ of vortex sheets (of a given topological type) in $M$, 
i.e. the space of hypersurfaces which bound diffeomorphic domains of a fixed volume in $M$.
Define the following (weak) metric  on the space $VS(M)$. A tangent vector to a point $\Gamma$ 
in the space  of all vortex sheets $VS(M)$ can be regarded as   a vector field $u$  
attached at the vortex sheet $\Gamma \subset M$ and normal to it. 
Then its square length is set to be
$$
\langle u,u\rangle_{vs}:=\inf\left\{\int_M (v,v)\mu \mid {\rm div} u=0, \mbox{ and } (v,n)\,n = u  \mbox{ on } \Gamma\right\},
$$
i.e. $v$ is a smooth divergence-free vector field in $M$, $n$ is the unit normal field to $\Gamma$,
and the normal component of $v$ is given by the field $u$ on $\Gamma$. 
Then the fluid flow with such vortex sheets satisfies the following variational principle:
Geodesics with respect to this metric  on  the space  $VS(M)$ 
describe the motion of  vortex sheets in an incompressible flow 
which is globally potential outside of the vortex sheet. 

To unify these two geodesic approaches, as well as to develop Arnold's approach to cover velocity fields 
with discontinuities, we introduce the Lie  groupoid of volume-preserving diffeomorphisms 
of a manifold $M$ that are discontinuous along a hypersurface, see details in \cite{IK17}.
\medskip

{\bf Theorem.}
{\it The Euler equation for a fluid flow with a vortex sheet $\Gamma\subset M$  is a geodesic
equation for the right-invariant $L^2$-metric on (source fibers of) the Lie groupoid 
of discontinuous volume-preserving diffeomorphisms.
}
\medskip

In particular, this leads to consideration of the following metric on the space $VS$ of closed hypersurfaces (vortex sheets) in $M$. Denote by $D^\pm$  components of $M\setminus \Gamma$. 
\medskip

{\bf Definition.}
For  a tangent vector $u$ to the base $VS$  this metric reads
$$
     \langle u, u \rangle_{vs} =  \int_{D^+_\Gamma}(\nabla f^+, \nabla f^+)\,\mu 
+  \int_{D^-_\Gamma}(\nabla f^-, \nabla f^-)\,\mu\,,
$$
where $\Delta f^\pm=0$ in $D^\pm_\Gamma$ and the normal component of $\nabla f^\pm$ at $\Gamma$ is $u$:
$\partial f^\pm/\partial n=(u,n)$.
Equivalently, upon integration by parts, we have 
$$
\langle u,u\rangle_{vs}= \sum \int_{\Gamma}f^\pm \,i_n\mu =   \int_{\Gamma} ({NtD}^+ + {NtD}^-) (u,n) \,i_u\mu\,, 
$$
where the Neumann-to-Dirichlet operators $NtD^\pm$ 
on the domains $D^\pm_\Gamma$  are regarded as operators on functions.

\end{talk}

\begin{talk}[St\'{e}phane Puechmorel]{Alice Le Brigant}
{Optimal Riemannian quantization}
{Le Brigant, Alice}

Optimal quantization is concerned with finding the best approximation of a probability distribution $\mu$ by a discrete measure $\mu_n$ supported by a given number $n$ of points \cite{graf}. While the case of probability distributions on a Riemannian manifold has recently received attention \cite{iacobelli}, most work on optimal quantization is suited for vector or functional data, and the Riemannian setting is lacking numerical schemes to compute the optimal approximation $\mu_n$.

\section{The optimal quantization problem}
Let $\mu$ be a probability measure on a complete Riemannian manifold $M$, with compact support. One way to express the optimal quantization problem is as follows: given $X$ a random variable with distribution $\mu$, approximate it by a \emph{quantized version} $q(X)$, where $q:M\rightarrow M$ maps $M$ to a set of at most $n$ points ($|q(M)|\leq n$) and minimizes the following $L^p$ cost
 \begin{equation}
 \label{e1}
\mathbb E_\mu\left(d(X,q(X))^p\right)
 \end{equation}
If $q(M) =\{a_1,\hdots,a_n\}$, then for all $x\in M$, $d(x,q(x))\geq \min_{1\leq i\leq n}d(x,a_i)$, with equality if and only if $q$ is the projection to the nearest neighbor of $q(M)$. Therefore, the mapping $q$ minimizing \eqref{e1} is necessarily of the form
\begin{equation}
\label{optq}
q(x) = \sum_{i=1}^n a_i \mathbf{1}_{C_i}(x), 
\end{equation}
where $C_i=\{ x\in M, d(x,a_i) \leq d(x,a_j) \, \forall j\neq i\}$ is the Voronoi cell associated to $a_i$. Note that if the support of $\mu$ has at least $n$ elements, one easily checks that $q$ can always be improved by adding an element to its image, and so the $a_i$'s are pairwise distinct. The optimal quantization problem can therefore be expressed as follows.
\begin{problem}
Given $\mu$, find $n$ points $a_1,\hdots,a_n$ that realize the infimum in
\begin{equation*}
e_{n,p}(\mu) = \inf_{a_1,\hdots,a_n}\mathbb E_\mu\left( \min_{1\leq i\leq n} d(X,a_i)^p\right).
\end{equation*}
\end{problem}\label{pb1}
Note that if $n=1$, the solution of Problem \ref{pb1} is simply the Riemannian center of mass. Just as in the vector case \cite{graf}, Problem \ref{pb1} has an equivalent formulation in terms of the $L^p$ Wasserstein distance $W_p$.
\begin{problem}\label{pb2}
Given $\mu$, find a discrete measure $\mu_n$ supported by $n$ points that minimizes the $L^p$ Wasserstein distance to $\mu$
\begin{equation*}
e_{n,p}(\mu) = \inf_{|\text{supp}\,\mu_n|= n} W_p(\mu,\mu_n).
\end{equation*}
\end{problem}
The solution $\mu_n$ of Problem \ref{pb2} is simply the image measure of $\mu$ by the optimal mapping \eqref{optq} of Problem \ref{pb1}, i.e. the atoms of $\mu_n$ are the points $a_1,\hdots,a_n$ of the image of $q$ and their weights are given by the $\mu$-mass of their Voronoi cells
\begin{equation*}
\mu_n = \sum_{i=1}^n\mu(C_i) \delta_{a_i}.
\end{equation*}

\section{Competitive Learning Riemannian quantization}
Of Problems \ref{pb1} and \ref{pb2}, the first one is the easiest to work with and so we focus on finding a numerical scheme to compute the minimizers of
\begin{equation*}
F_{n,p}(a_1,\hdots,a_n)=\mathbb E_\mu\left( \min_{1\leq i\leq n} d(X,a_i)^p\right).
\end{equation*}
Since we have assumed that $\mu$ has compact support, the existence of a minimizer $\alpha=(a_1,\hdots,a_n)$ is easily obtained \cite{lebrigant}. The minimizer $\alpha$ is in general not unique, first of all because any permutation of $\alpha$ is still a minimizer, and secondly because any symmetry of $\mu$, if it exists, will transform $\alpha$ into another minimizer of $F_{n,p}$ (consider e.g. the uniform distribution on the sphere). Just as in the vector case, the cost function $F_{n,p}$ is differentiable and its gradient is given, for $p=2$ (a similar expression is easily obtained for any $p$), by \cite{lebrigant}
\begin{equation}\label{grad}
\nabla_\alpha F_{n,2} = -2\left( \mathbb E_\mu \mathbf{1}_{\{X\in \mathring{C_i}\}} \overrightarrow{a_iX}\right)_{1\leq i\leq n}.
\end{equation}
Note that any minimizer $\alpha$ verifies $\int_{\mathring{C_i}} \overrightarrow{a_ix}\,\mu(dx) = 0$ for all $i$ and so each $a_i$ is the $\mu$-center of mass of its Voronoi cell $C_i$. \emph{Competitive Learning Quantization} is a stochastic gradient descent method to compute a minimizer $(a_1,\hdots,a_n)$ of $F_{n,2}$ using an online sequence of i.i.d. observations $X_1, X_2, \hdots$ sampled from $\mu$. At each step $k$, the algorithm receives an observation $X_k$ and the current points $a_1(k),\hdots,a_n(k)$ are moved in the opposite direction of the gradient, which is approximated by
\begin{equation*}
\left(\mathbf{1}_{\{X_k\in \mathring{C_i}\}} \overrightarrow{a_i(k)X_k}\right)_{1\leq i\leq n}.
\end{equation*}
Therefore, the only point $a_i(k)$ to be moved at step $k$ is the closest neighbor, among all the $a_j(k)$'s, of the new observation $X_k$, and it is updated in the direction of that new observation using a step $\gamma_k$
\begin{equation*}
a_i(k+1) = \exp_{a_i(k)}\left(\gamma_k \overrightarrow{a_i(k)X_k}\right).
\end{equation*}
The other $a_j(k)$'s stay unchanged. Using a theorem from Bonnabel \cite{bonnabel2013}, this algorithm can be shown to converge \cite{lebrigant}. Applied to $N$ data points $x_1,\hdots, x_N$, it yields an optimal approximation (or summary) of their empirical measure $\mu=\frac{1}{N}\sum_{i=1}^N \delta_{x_i}$ by a discrete measure supported by a smaller number $n \ll N$ of points $a_1,\hdots,a_n$, which correspond to the centers of a $K$-means type clustering.

\section{Application to air traffic analysis}

In air traffic analysis, Competitive Learning Riemannian Quantization can be used to compute summaries of images of covariance matrices estimated from air traffic images \cite{lebrigant}. These summaries are representative of the air traffic complexity and yield clusterings of the airspaces into zones that are homogeneous with respect to that criterion, as shown in Figure \ref{fig:france}. They can then be compared using discrete optimal transport and be further used as inputs of a machine learning algorithm or as indexes in a traffic database.

\begin{figure}[H]
\includegraphics[width=0.3\textwidth,height=0.3\textwidth]{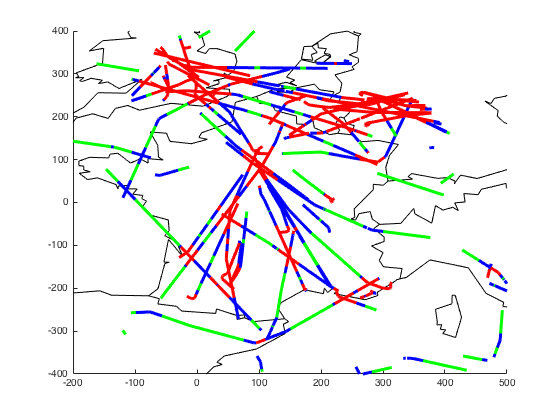}
\includegraphics[width=0.3\textwidth,height=0.3\textwidth]{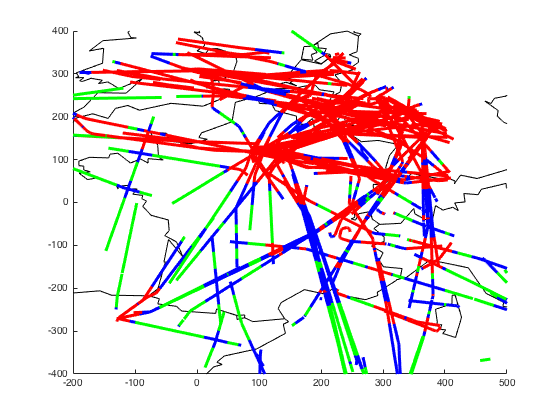}
\includegraphics[width=0.3\textwidth,height=0.3\textwidth]{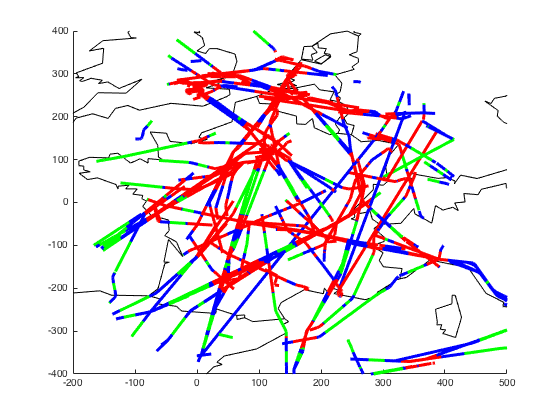}
\caption{Clustering of the French airspace using Competitive Learning Riemannian Quantization.}
\label{fig:france}
\end{figure}

\end{talk}

\begin{talk}[Klas Modin]{Stephen Marsland}
{Horizontal stochastic gradient flows for the Hopf fibration}
{Marsland, Stephen}

\section{Introduction}

We use numerical computations to explore horizontal Brownian motion and stochastic gradient flows relative to the Hopf fibration. The study of this problem is motivated by an infinite-dimensional analogue, where the Hopf fibration is replaced by Moser's fibration of  $\operatorname{Diff}(M)$ as an $\operatorname{SDiff}(M)$ 
bundle over the space  $\operatorname{Dens}(M)$ of probability densities.
  
\section{Stochastic motion on  $\mathcal{S}^2$}

Our horizontal Brownian motion on the unit quaternions  $\mathcal{S}^3$ is generated by horizontal lifting of a uniform Brownian motion on $\mathcal{S}^2$. Thus, the first step is to generate a (standard) Brownian motion on $\mathcal{S}^2$.

The form of our stochastic flow is
\begin{equation}
    d \mathbf{r} = - \nabla f(\mathbf{r} dt + dw
\end{equation}
where $f \in \mathcal{S}^2 \subset \mathbb{R}^3$ and $f$ is a $C^2$ function on $S^2$.

We start by generating the gradient function and the noise. 
Notice that we generate the noise as standard Gaussian in $\mathbb{R}^3$.
The corresponding noise vector at $T_\mathbf r S^2$ is given by orthogonal projection onto the plane defined by $\mathbf r$.
 
 We can now generate the stochastic motion on $\mathcal{S}^2$. For simplicity we use Euler's explicit method along geodesics (great circles). Initial positions are given by the north pole $\mathbf r_0 = (0,0,1)$.
 
 The points are initially clustered around the origin (north pole), but tend to spread out to fill the entire sphere. 
To understand the diffusion better we compute the distance of each of the points to the north pole, and we look at the corresponding time series.

For small $t$ we expect the average to approximately behave as $\sqrt{t}$ since $S^2$ is locally an approximation of $\mathbb{R}^2$, and this is confirmed in Figure~\ref{fig:curve1}.

 We also compute a uniform random walk on $\mathcal{S}^3$ to show that these cover more of the surface, as can be seen in Figure~\ref{fig:fills}.

\begin{figure}[h]
    \centering
    \includegraphics[scale=0.5]{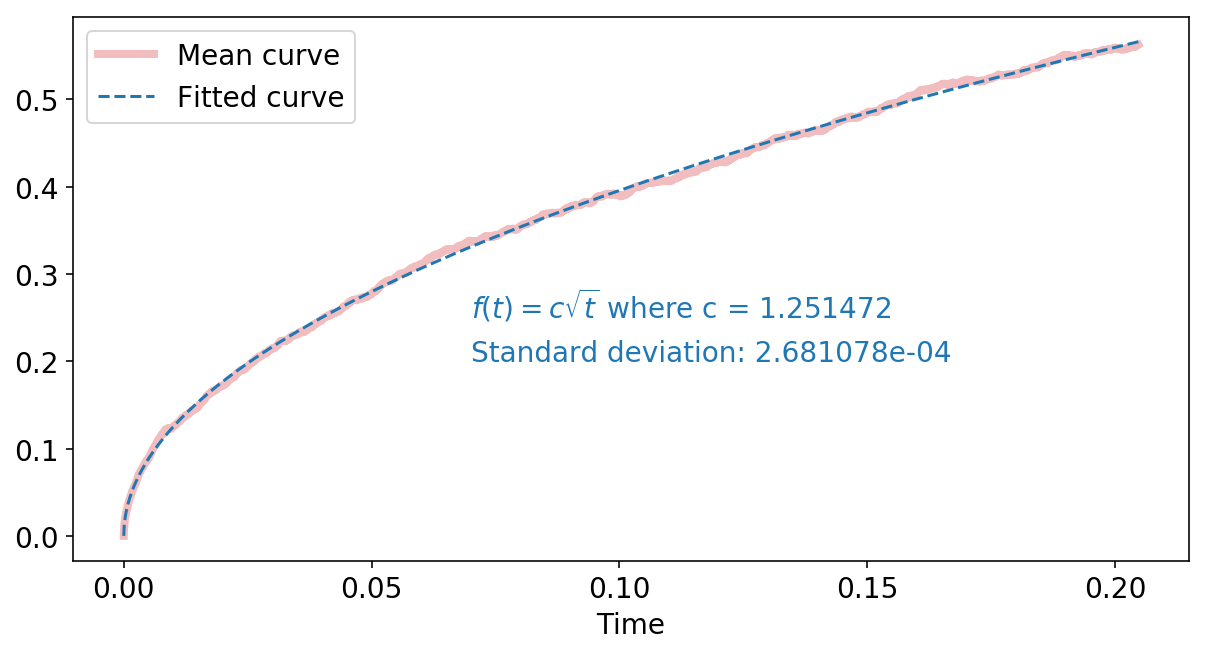}
    \caption{The average curve looks very much like $\sqrt t$; best fit and model are shown.}
    \label{fig:curve1}
\end{figure}
 
 \begin{figure}[h]
    \centering
    \includegraphics[scale=0.25]{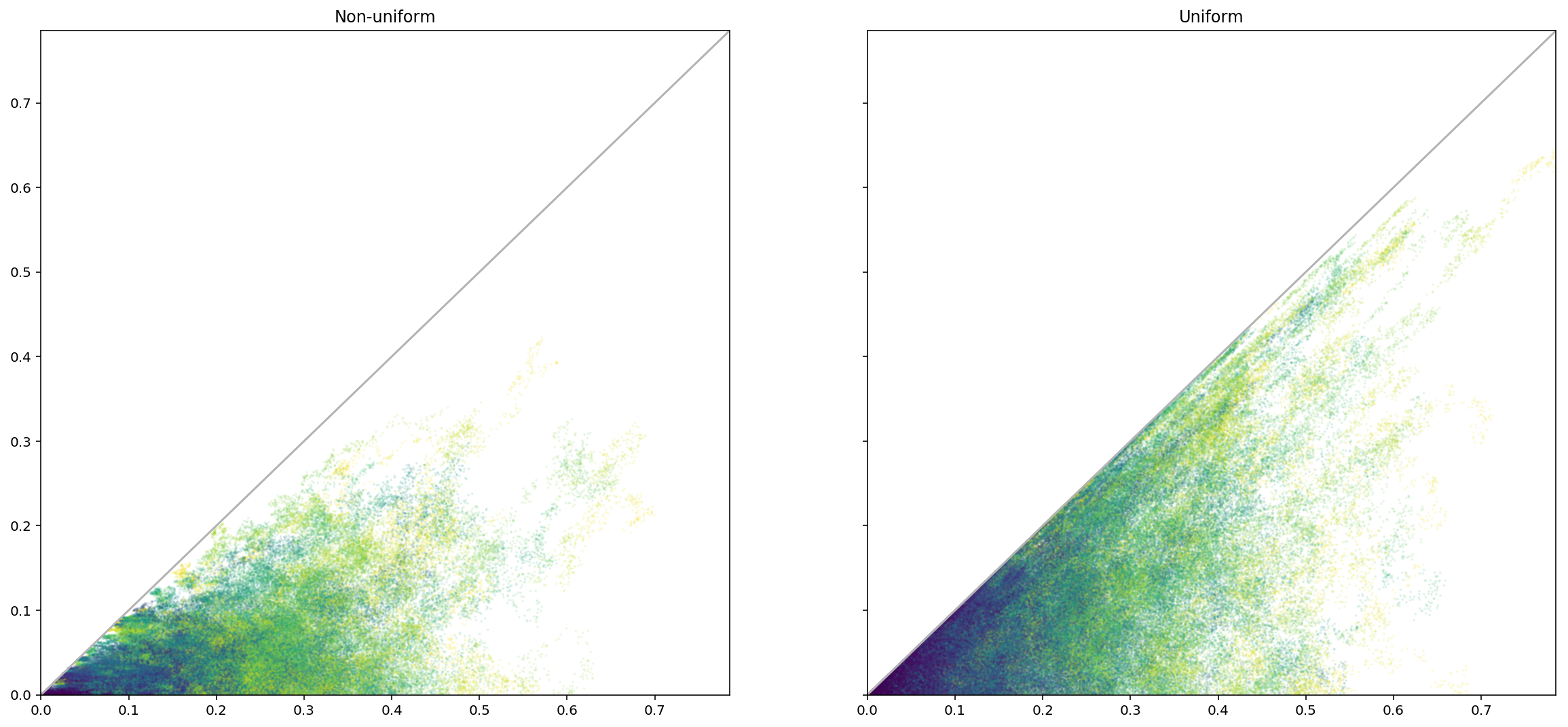}
    \caption{Scatter plots of the coverage of the sphere for 100 random walks for the non-uniform (left) and uniform (right) travelling up to $\pi/4$ from the origin.}
    \label{fig:fills}
\end{figure}

\end{talk}

\clearpage

\begin{talk}{Peter Michor}{All diffeomorphism invariant tensor fields on the space of smooth positive densities on a compact manifold with corners}{Michor, Peter W.}

\section{Introduction} 

The Fisher--Rao metric on the space $\operatorname{Prob}(M)$ of probability densities is invariant under the action of the diffeomorphism group $\operatorname{Diff}(M)$.
Restricted to finite-dimensional submanifolds of $\operatorname{Prob}(M)$, 
so-called statistical manifolds, it is called Fisher's information metric \cite{Ama1985}. 
A uniqueness result was established \cite[p. 156]{Cen1982} for Fisher's information 
metric on finite sample spaces and \cite{AJLS2014} extended it to infinite sample spaces.       
The Fisher--Rao metric on the infinite-dimensional manifold of all positive probability densities 
was studied in \cite{Fri1991}, including the computation of its curvature. 
In \cite{BBM2016} it was proved that the Fisher--Rao metric on $\operatorname{Prob}(M)$ is, up to a multiplicative constant, the unique $\operatorname{Diff}(M)$-invariant metric, on a compact manifold without boundary. In fact, all $\operatorname{Diff}(M)$-invariant bilinear tensor fields on the space $\operatorname{Dens}_+(M)$ of all positive smooth densities were determined. Here we try to extend this result to compact smooth manifolds with corners and we also determine all $\operatorname{Diff}(M)$-invariant tensor fields on $\operatorname{Dens}_+(M)$ of all orders. 
The changes required in the proof are quite subtle.

\subsection{The Fisher--Rao metric}
Let $M^m$ be a smooth compact connected manifold without boundary. 
Let $\operatorname{Vol}(M)\to M$ be the 
the line bundle of smooth densities whose cocycle of transition functions is 
$|\det(d(u_a\circ u_b^{-1}))|^{-1}\circ u_b$ for any smooth atlas $(u_a:U_a \to \mathbb R^m)_{a\in A}$;  for more details 
 we refer to \cite{BBM2016}. Moreover, we denote by $|\quad|:\Lambda^mT^*M\to \operatorname{Vol}(M)$ the fiber respecting absolute value mapping. 
We let $\operatorname{Dens}_+(M)$ denote the space of smooth positive densities on $M$, i.e., 
$\operatorname{Dens}_+(M) = 
\{ \mu \in \Gamma(\operatorname{Vol}(M)) \,:\, \mu(x) > 0\; \forall x \in M\}$. Let $\operatorname{Prob}(M)$ be the 
subspace of positive densities with integral 1 on $M$. Both spaces are smooth Fr\'echet manifolds, 
in particular they are open subsets of the affine spaces of all densities or densities of integral 
1, respectively. For $\mu \in \operatorname{Dens}_+(M)$ we have    
$ T_\mu \operatorname{Dens}_+(M) = \Gamma(\operatorname{Vol}(M))$ and for $\mu\in \operatorname{Prob}(M)$ we have 
$$
T_\mu\operatorname{Prob}(M)=\{\alpha\in \Gamma(\operatorname{Vol}(M)): \int_M\alpha =0\}.
$$
The Fisher--Rao metric is a Riemannian metric on $\operatorname{Prob}(M)$ and is defined as follows:
$$
G^{\operatorname{FR}}_\mu(\alpha,\beta) = \int_M \frac{\alpha}{\mu}\frac{\beta}{\mu}\mu.
$$
This metric is invariant under the associated action of $\operatorname{Diff}(M)$ on $\operatorname{Prob}(M)$, since
$$
\Big((\varphi^*)^*G^{\operatorname{FR}}\Big)_\mu(\alpha,\beta) = G^{\operatorname{FR}}_{\varphi^*\mu}(\varphi^*\alpha,\varphi^*\beta) 
= \int_M \Big(\frac{\alpha}{\mu}\circ \varphi\Big)\Big(\frac{\beta}{\mu}\circ \varphi\Big)\varphi^*\mu
= \int_M \frac{\alpha}{\mu}\frac{\beta}{\mu}\mu\,.
$$

The uniqueness result for the Fisher--Rao metric follows from the following classification of 
$\operatorname{Diff}(M)$-invariant bilinear forms on $\operatorname{Dens}_+(M)$.

\begin{theorem*} {\rm \cite{BBM2016}} Let $M$ be a compact manifold without boundary of dimension $\geq 2$.
Let $G$ be a smooth (equivalently, bounded) bilinear form on $\operatorname{Dens}_+(M)$ which is invariant under the action of 
$\operatorname{Diff}(M)$. Then 
$$
G_\mu(\alpha,\beta)=C_1(\mu(M)) \int_M \frac{\alpha}{\mu}\frac{\beta}{\mu}\,\mu + C_2(\mu(M))  \int_M\alpha \cdot \int_M\beta
$$
for some smooth functions $C_1,C_2$ of the total volume $\mu(M)$. 
\end{theorem*}

\subsection{Acknowledgments} We thank Boris Kruglikov, Valentin Lychagin, Philipp Harms, Armin Rainer, Andreas Kriegl for discussions and hints.  

\section{Manifolds with corners}

\subsection{Manifolds with corners alias quadrantic (orthantic) manifolds}
For more information we refer to \cite{DouadyHerault73}, \cite{Michor80}, \cite{Melrose96}, etc.
Let $Q=Q^m=\mathbb R^m_{\ge 0}$ be the positive orthant or quadrant. By Whitney's extension theorem or Seeley's theorem,
restriction $C^{\infty}(\mathbb R^m)\to C^{\infty}(Q)$ is a surjective continuous linear mapping which admits a continuous linear section (extension mapping); so $C^{\infty}(Q)$ is a direct summand in $C^{\infty}(\mathbb R^m)$. A point $x\in Q$ is called a \emph{corner of codimension} $q>0$ if $x$ lies in the intersection of $q$ distinct coordinate hyperplanes. Let $\partial^q Q$ denote the set of all corners of codimension $q$.

A manifold with corners (recently also called a quadrantic manifold) $M$ 
is a smooth manifold modelled on open subsets of $Q^m$.
We assume that it is connected and second countable; then it is paracompact and for each open cover it admits a subordinated smooth partition of unity. Any manifold with corners $M$ is a submanifold with corners of an open manifold $\tilde M$ of the same dimension, and each smooth function on $M$ extends to a smooth function on $\tilde M$. Moreover,  restriction $C^\infty(\tilde M)\to C^\infty(M)$ is a surjective continuous linear map which admits a continuous linear section; this follows by  gluing via a smooth partition of unity  from the result about quadrants.Thus $C^\infty(M)$ is a topological direct summand in 
$C^\infty(\tilde M)$ and the same holds for the dual spaces: The space of distributions 
$\mathcal D'(M)$, which we identity with $C^\infty(M)'$ in this paper, is a direct summand in 
$\mathcal D'(\tilde M)$. It consists of all distributions with support in $M$.

We do not assume that $M$ is oriented, but eventually we will assume that $M$ is compact. 
Diffeomorphisms of $M$ map the boundary $\partial M$ to itself and map  the boundary $\partial^q M$ of corners of codimension $q$ to itself; $\partial^q M$ is a submanifold of codimension $q$ in $M$; in general $\partial^q M$ has finitely many connected components. We shall consider $\partial M$ as stratified into the connected components of all $\partial^q M$ for $q>0$.

Each diffeomorphism of $M$ restricts to a diffeomorphism of $\partial M$ and to a diffeomorphism of each $\partial^q M$. The Lie 
algebra of $\operatorname{Diff}(M)$ consists of all vector fields $X$ on $M$ such that $X|\partial^q M$ is tangent to 
$\partial^q M$. We shall denote this Lie algebra by $\mathfrak X(M,\partial M)$.

\subsection{Differential forms}
There are several differential complexes on a manifold with corners.  
If $M$ is not compact there are also the versions with compact support. 
\begin{itemize}
\item Differential forms that vanish near $\partial M$. If $M$ is compact, this is the same as
the differential complex $\Omega_c(M\setminus \partial M)$ of differential forms with compact support 
in the open interior $M\setminus \partial M$. 
\item $\Omega(M,\partial M) = \{\alpha\in \Omega(M): j_{\partial^q M}^*\alpha =0 \text{ for all } q\ge 1\}$, the complex of differential forms that pull back to 0 on each boundary stratum. 
\item $\Omega(M)$, the complex of all differential forms. Its cohomology equals 
singular cohomology with real coefficients of $M$, since $\mathbb R\to \Omega^0\to \Omega^1\to \dots$
is a fine resolution of the constant sheaf on $M$; for that one needs existence of smooth partitions of unity and the Poincar\'e lemma which holds on manifolds with corners.
The Poincar\'e lemma can be proved as in \cite[9.10]{Mic2008} in each quadrant.
\end{itemize}
If $M$ is an oriented manifold with corners of dimension $m$ and if $\mu\in \Omega^m(M)$ is a nowhere vanishing form of top degree, then $\mathfrak X(M)\ni X\mapsto i_X\mu\in \Omega^{m-1}(M)$ is a linear isomorphism. 
Moreover, $X\in \mathfrak X(M,\partial M)$ (tangent to the boundary) if and only if $i_X\mu\in\Omega^{m-1}(M,\partial M)$.

Let us consider the short exact sequence of differential graded algebras
$$
0\to \Omega(M,\partial M) \to \Omega(M) \to \Omega(M)/\Omega(M,\partial M)\to 0\,.
$$
The complex $\Omega(M)/\Omega(M,\partial M)$ is a subcomplex of the product  of $\Omega(N)$ for all connected components $N$ of all 
$\partial^q M$. The quotient consists of forms which extend continuously over boundaries to $\partial M$ with its induced topology in such a way that one can extend them to smooth forms on $M$; this is contained in the space of `stratified forms' as used in \cite{Valette15}. There Stokes' formula is proved for stratified forms.

\begin{proposition}[Stokes' theorem] \label{Stokes}
For a connected oriented manifold $M$ with corners of dimension $\dim(M)=m$ and for any $\omega\in\Omega^{m-1}_c(M)$ we have
$$
\int_M d\omega = \int_{\partial^1M} j_{\partial^1 M}^*\omega\,.
$$
\end{proposition}

See \cite{BMPR18} for a short proof.

\subsection{Top cohomology of the pair \texorpdfstring{$(M,\partial M)$}{(M, dM)}}
For a connected oriented manifold with corners $M$ of dimension $m$ (we assume that $\partial M$ is not empty) we consider the following diagram; see \cite[section 8]{BMPR18} for the simple proofs. 
Here $\omega\in\Omega^m_c(M\setminus \partial M)$ is a fixed form with $\int\omega = 1$.
All instances of $\mathbb R$ in the diagram are connected by identities which fit commutingly into the diagram.
Each line is the definition of  the corresponding top de~Rham cohomology space. 
The integral in the first line induces an isomorphism in cohomology since 
$M\setminus \partial M$ is a connected oriented open manifold.
The bottom triangle commutes by Stokes' theorem \ref{Stokes}.
$$
\xymatrix{
\Omega^{m-1}_c(M\setminus \partial M) \ar[r]^{d} \ar@{^{(}->}[d] & 
\Omega^m_c(M\setminus \partial M) \ar@/^1.5pc/[rr]^{\int_{M\setminus \partial M}} \ar@{->>}[r]  \ar@{^{(}->}[d] &
H^m_c(M\setminus \partial M) \ar@{=}[r]  \ar[d]& \mathbb R
\\
\Omega^{m-1}_c(M, \partial M) \ar[r]^{d} \ar@{^{(}->}[d] & 
\Omega^m_c(M,\partial M) \ar@{->>}[r] \ar@{=}[d] \ar@/_1.5pc/[rr]_{\qquad\qquad\qquad\qquad\quad\int_M} &
H^m_c(M,\partial M) \ar@{=}[r] \ar[d]&  \mathbb R
\\
\Omega^{m-1}_c(M) \ar[r]^{d} \ar[dr]_{\int_{\partial^1 M}\circ j_{\partial^1 M}^*} & 
\Omega^m_c(M)  \ar@{->>}[r] \ar[d]^{\int_M}  &
H^m_c(M) \ar@{=}[r]& 0
\\
& \mathbb R  
}
$$

\begin{theorem} {\rm (\cite{BMPR18} Moser's theorem for manifolds with corners)}\label{Moser}
Let $M$ be a compact smooth manifold with corners, possibly non-orientable. 
Let $\mu_0$ and $\mu_1$ be two smooth positive densities in 
$\operatorname{Dens}_+(M)$ with $\int_M\mu_0 = \int_M\mu_1$.
Then there exists a diffeomorphism $\varphi:M\to M$ such that 
$\mu_1= \varphi^*\mu_0$. If  and only if $\mu_0(x)=\mu_1(x)$ for each corner $x\in\partial^{\ge 2}M$ 
of codimension $\ge 2$, then $\varphi$ can be chosen to be the identity on $\partial M$. 
\end{theorem}

\section{Diffeomorphism invariant tensor fields on the space of densities}

\begin{conjecture}
Let $M$ be a compact manifold with corners, of dimension $m\ge2$, and let 
$$\partial^pM=(\partial^pM)_1 \sqcup (\partial^pM)_2\sqcup \dots\sqcup (\partial^pM)_{n_p}$$
be the decomposition of the set of corners of codimension $p$ into its connected components which are manifolds of dimension $m-p$. 

Then the the associative algebra of bounded $\Diff_0(M)$-invariant tensor fields on $\operatorname{Dens}_+(M)$ is has the following set of generators, where $\mu\in \operatorname{Dens}_+(M)$ is the footpoint and $\alpha_i\in \Gamma(\operatorname{Vol}(M)) = T_\mu\operatorname{Dens}_+(M)$: 
\begin{align*}
&f(\mu(M)),\qquad\text{ where }f\in C^\infty(\mathbb R_{>0},\mathbb R)
\\&
\int_M \alpha_1 =  \int_M\frac{\alpha_1}{\mu}\,\mu
\\&
\int_M \frac{\alpha_1}{\mu}\dots\frac{\alpha_n}{\mu}\,\mu \quad n\ge2
\\&
\int_{(\partial^p M)_j} \frac{\alpha_1}{\mu}d\Big(\frac{\alpha_{i_2}}{\mu}\Big)\wedge \dots\wedge  d\Big(\frac{\alpha_{i_{m+1}}}{\mu}\Big)
\\&
\int_{(\partial^p M)_j} \frac{\alpha_1}{\mu}d\Big(\frac{\alpha_{i_2}}{\mu}\Big)\wedge \dots\wedge  d\Big(\frac{\alpha_{i_{m+1-p}}}{\mu}\Big), \quad p=1,\dots m-1, j=1,\dots, n_p 
\\&   
\frac{\alpha}{\mu}((\partial^mM)_j),  \quad\text{ for  }j=1,\dots,n_m;\quad \text{ note that }(\partial^mM)_j)\text{ is a point.}
\end{align*}  
\end{conjecture}

Parts of the proof of the conjecture are already written and okay, with other parts I am still fighting.

\end{talk}

\begin{talk}{Xavier Pennec}
{Taylor expansions, Fr\'echet mean of an $n$-sample and CLT in affine Manifolds}{Pennec, Xavier}

\noindent

\section{Taylor expansions in manifolds}

Discrete methods for parallel transport in manifolds based on geodesics like Schild's ladder are regularly used in applications, for instance to transport longitudinal organ deformations from one subject to another in computational anatomy. However, analyzing the numerical accuracy of these algorithms requires to perform Taylor expansions of functions and vector fields in manifolds. In Riemannian manifolds, the Taylor expansion of the Riemannian metric in a normal coordinate system is well known for the first orders. 
However, going to orders higher than 4 is computationally much more involved. 

In a recent work, Gavrilov \cite{gavrilov_algebraic_2006,gavrilov_double_2007},
developed a very general expansion of the composition of two exponential maps in the spirit of the Baker-Campbell-Hausdorff (BCH) formula for Lie groups that holds in affine connection manifolds: 
the double exponential $\exp_x(v,u) = \exp_{\exp_x(v)}( \Pi_x^{\exp_x(v)} u)$ corresponds to a first geodesic shooting from the point $x$ along the vector $v$, followed by a second geodesic shooting from $y= \exp_x(v)$ along the parallel transport $\Pi_x^y u$ of the vector $u$ along the first geodesic.  For a symmetric affine connection, the Taylor expansion of the log of this endpoint $h_x(u,v) = \log_x(\exp_x(v,u))$  is:
\begin{equation}
\begin{split}
h_x(v, u) =\,  &v + u + \frac{1}{6}R(u,v)v + \frac{1}{3}R(u,v)u 
+ \frac{1}{24}\nabla_v R(u,v) (2v +5u)  \\ &+ \frac{1}{24}(\nabla_u R)(u,v) (v + 2u)
+O(5),
\end{split}
\end{equation}
where  $O(5)$ represents polynomial terms of order 5 or more in u and v.  
This surprisingly simple formulation allows to analyze very simply some of the discrete parallel transport algorithm and to analyze the impact of the curvature on the empirical Fr\'echet mean.

One can easily come back to the Riemannian setting to compute the Taylor expansion of the squared geodesic distance between two points $x_v = \exp_x(v)$ and $x_w = \exp_x(w)$ that are close to $x$:
\begin{equation}
\begin{split}
\text{dist}^2(x_v, x_w) = \, & \| w - v \|^2_x  + \frac{1}{3} \left< R(w,v)w \, , \, v \right>_x
+\frac{1}{12}  \left< \nabla_w R(w,v)w\, , \, v\right>_x  
\\ & + \frac{1}{12}  \left< \nabla_v R(w,v)w\, , \, v\right>_x 
+O(6)  .
\end{split}
\end{equation}

\section{Analysis of the pole ladder parallel transport algorithm}

Pole ladder is a variation of Schild's ladder to transport along a geodesic which is used as one of the edges of the geodesic parallelogram constructed to approximate the transport. One step of pole ladder transport of  the vector $u_p$ at $p$ to $u_q$ at $q$ along the geodesic segment $[p,q]$ can written as a composition of two geodesic symmetries: $u_q = s_q(s_m(\exp_p(u_p)))$, where $s_p(q) = \exp(-\log_p(q))$ and $m$ is the mid-point of $[p,q]$. Steps of Schild's and pole ladders were shown to be first order approximation of the parallel transport $\Pi_p^q u_q$. Computing higher order terms for the pole ladder turn out to be quite simple thanks to Gavrilov's expansion. Let $u = \Pi_p^m u_p$ (resp. $u' = \Pi_q^m u_q$) be the parallel transport of $u_p$ at $p$ (resp. $u_q$ at $q$) to the mid-point $m$ along the geodesic segment $[p,q] = [\exp_m(-v),\exp_m(v)]$. We find that the error on one step of pole ladder is:
\begin{equation}
u' = u + \frac{1}{12} (\nabla_v R)(u,v)(5 u - 2v) + \frac{1}{12}  (\nabla_{u} R)(u,v)(v -2 u) +O(5).
\end{equation}

We see that each step of pole ladder is a third order approximation: this is much better than expected. Moreover, the fourth order error term vanish in affine symmetric spaces because the curvature is covariantly constant. In fact, one can actually prove that all error terms vanish in a convex normal neighborhood of an affine connection space: one step of pole ladder realizes an exact parallel transport (provided that geodesics and mid-points are computed exactly of course) \cite{pennec2018parallel}. This result makes pole ladder a very attractive scheme for Lie groups endowed with their canonical  symmetric space structure (the symmetric Cartan-Schouten connection).

\section{Fr\'echet mean of an \texorpdfstring{$n$}{n}-sample and CLT in affine Manifolds}

Let $\mu$ be a probability density in an affine manifold whose support is included in a sufficiently small geodesic ball $U$ of radius $\epsilon$ (measured with an auxiliary metric) such that the exponential barycenter is unique in this ball (Karcher/Kendall uniqueness conditions for the Fr\'echet mean in Riemannian manifolds). The $k$-moment of this distribution is the $k$-covariant tensor field  $\mathfrak{M}_k(x) = \int_U \log_x(y) \otimes^k \mu(dy)$.  
The first moment $\mathfrak{M}_1(x)$ is a vector field whose zero is the unique exponential barycenters / Fr\'echet mean $\bar x$. The second moment $\mathfrak{M}_2(x)$ is an SPD matrix field whose value at $\bar x$ is the covariance matrix. 

Using Gavrilov's  double exponential, we can compute the series expansion of $\mathfrak{M}_1(\exp_x(v))$in the neighborhood of any point $x$.  
Solving for the value of $v$ that zeros out the first moment leads to the approximation of the field $\log_x(\bar x)$ pointing from each point to the mean $\bar x$:
 \begin{equation}
\label{LogFrechet}
\begin{split}
\log_x(\bar x) = &
\mathfrak{M}_1  - \frac{1}{3} R(\bullet, \mathfrak{M}_1)\bullet : \mathfrak{M}_2 
 \frac{1}{24} \nabla_\bullet R(\bullet,\mathfrak{M}_1)\mathfrak{M}_1 : \mathfrak{M}_2 
\\ & -\frac{1}{8} \nabla_{\mathfrak{M}_1} R(\bullet, \mathfrak{M}_1) \bullet : \mathfrak{M}_2
-\frac{1}{12} \nabla_\bullet R(\bullet, \mathfrak{M}_1) \bullet :\mathfrak{M}_3 + O(\epsilon^5),
\end{split}
\end{equation} 
where the notation $R(\bullet, \mathfrak{M}_1)\bullet : \mathfrak{M}_2$ denotes the contraction of the curvature tensor $R$ along the axes specified by the dots.

An $n$-sample  $\{x_1, \ldots x_n\}$ of IID variables can be encoded as the discrete probability distribution $X_n = \frac{1}{n} \sum_{i=1}^n \delta_{x_i}$. The moments of this distribution are the $k$-covariant tensor fields  $\hat {\mathfrak{M}}^n_k(x) = \sum_{i=1}^n\log_x(x_i)\otimes^k$.  The empirical Fr\'echet mean $\bar x_n$ of this sample solves $\hat {\mathfrak{M}}^n_1(\bar x_n) =0$. Because the $n$-sample is a random variable, the moments are random tensor fields and the empirical Fr\'echet mean of an $n$-sample is a random variable on the manifold. 
The random vector $z_n = \log_{\bar x}(\bar x_n)$ quantifies how much it differs from the Fr\'echet mean $\bar x$ of the underlying distribution. To characterize the law of this random vector, we have to compute the expectation  of the empirical moments at the point $\bar x$.
Using formula \ref{LogFrechet}, which is valid for singular distributions, we get a series expansion of $\log_{\bar x}(\bar x_n)$ expressed in terms of empirical moments $\hat {\mathfrak{M}}_k$.
However, the expectation (w.r.t. the product law $\mu^{\otimes n}$) of an empirical moment does not exactly commutes with the tensor product of two (or more) empirical moments. We have indeed:
\[
{I\!\!E}( \hat {\mathfrak{M}}^n_k ) = {\mathfrak{M}}^n_k
\quad \text{and} \quad 
{I\!\!E}( \hat {\mathfrak{M}}^n_k \otimes \hat {\mathfrak{M}}^n_l ) = \frac{n-1}{n} {\mathfrak{M}}^n_k \otimes {\mathfrak{M}}^n_l + \frac{1}{n} {\mathfrak{M}}^n_{k+l}.
\]
Because $\mathfrak{M}_1$ vanishes at the Fr\'echet mean, many terms disappear, but the higher order moment in the above formula generates non-vanishing terms. 

The first moment of $z_n$ is the bias of the empirical Fr\'echet mean:
 \begin{equation}
\label{Bias}
b_n = {I\!\!E}( \log_{\bar x}(\bar x_n)) = {I\!\!E}( \hat {\mathfrak{M}}^n_1(\bar x) ) =
\frac{1}{6n} \mathfrak{M}_2: \nabla R : \mathfrak{M}_2 + O\left(\epsilon^5, \frac{1}{n^2}\right).
\end{equation}
Surprisingly, this bias is in $1/n$ and does not disappear with the classical scaling by $\sqrt{n}$ for a finite number of samples.
This bias is a double contraction of  the covariant derivative of the Riemannian curvature. It is thus vanishing for symmetric spaces which may explain why it has not been noticed so far. For a smooth  affine manifold with finite curvature and curvature gradient, the bias vanishes asymptotically so that the empirical mean remains a consistent estimator. However, in the limit case of  a manifold concentrating the curvature at singularities (e.g. corners), the curvature gradient is going to infinity and the empirical mean may become inconsistent. We conjecture that this phenomenon may explain sticky means  and repulsive means in stratified spaces.

The second moment of $z_n$ is the expected covariance matrix of $\bar x_n$: 
\begin{equation}
\label{Covariance}
\Sigma_n  
= {I\!\!E}( \hat {\mathfrak{M}}^n_2(\bar x) ) 
= \frac{1}{n} \mathfrak{M}_2 + \frac{1}{3n} \mathfrak{M}_2 : R : \mathfrak{M}_2 + O\left(\epsilon^5, \frac{1}{n^2}\right).
\end{equation} 
The classical rate of convergence $\frac{1}{n}\mathfrak{M}_2$ toward the mean is modified by a curvature term $\mathfrak{M}_2 : R : \mathfrak{M}_2$ which is also in $1/n$ which  accelerates the convergence rate with respected to the Euclidean case in negatively curved spaces, and slows  it down in positively curved spaces. This result is in line with the covariance formula of \cite{bhattacharya_statistics_2008} when  we perform the Tauylor expansion of the Hessian of the variance to make the curvature  appearing.

\end{talk}

\begin{talk}{Stephen Preston}{Solar models for 1D Euler-Arnold equations}{Preston, Stephen}

For $T\in (0,\infty]$, we consider a function $u(t,x)$ for $t\in [0,T)$ and $x\in \mathbb{S}^1 = \mathbb{R}/\mathbb{Z}$,
satisfying a PDE of the form
\begin{equation}\label{generalizedeuler}
Lu_t + uLu_x + 2u_x Lu = 0,
\end{equation}
where $L$ is some pseudodifferential operator.
This is an Euler-Arnold equation. Examples of primary interest include:
\begin{itemize}
\item Hunter-Saxton equation: $Lu=-u_{xx}$. 
\item Wunsch equation: $Lu = Hu_x$, where $H$ is the Hilbert transform.
\end{itemize}
We assume the mean of $u$ is zero: $\int_0^1 u(t,x)\,dx = 0$ for all $t$.
For any such equation there is a Lagrangian flow $\eta(t,x)$ defined by solving the equation
\begin{equation}\label{lagrangianflow}
\eta_t(t,x) = u\big(t,\eta(t,x)\big), \qquad \eta(0,x) = x.
\end{equation} 
The momentum is $\omega = Lu$, and it satisfies the conservation law
\begin{equation}\label{vorticityconservation}
\eta_x(t,x)^2 \omega\big(t,\eta(t,x)\big) = \omega_0(x), \quad \forall t\ge 0.
\end{equation}

\begin{theorem}
Consider the Hunter-Saxton equation
$u_{txx} + uu_{xxx} + 2u_xu_{xx}=0$
on the circle with Lagrangian flow as in \eqref{lagrangianflow}. Let $\rho(t,x) = \sqrt{\eta_x(t,x)}$. 
Let $\Gamma(t,x) = \big( \begin{smallmatrix} \rho(t,x) \\ -2\rho_x(t,x)\end{smallmatrix}\big)$; then $\Gamma$ satisfies the equation
\begin{equation}\label{HSsolar}
\Gamma_{tt}(t,x) = -K^2 \Gamma(t,x), \qquad \Gamma(0,x) = \Big( \begin{matrix} 1 \\ 0\end{matrix}\Big), \qquad \Gamma_t(0,x) = \Big( \begin{matrix} \tfrac{1}{2}u_0'(x) \\ \omega_0(x) \end{matrix}\Big),
\end{equation}
where $K^2 = \tfrac{1}{4} \int_0^1 u_0'(x)^2 \,dx$.

Furthermore the (constant) angular momentum of this central force system is given by $\Gamma_1 \dot{\Gamma}_2 - \Gamma_2 \dot{\Gamma}_1 = \omega_0(x)$.
Solutions of this system for $\Gamma$ exist globally, and we have $\eta_x = (\Gamma_1)^2$.
The Lagrangian flow leaves the diffeomorphism group at time $T$ when $\eta_x$ reaches zero, which first happens at a value of $x_0$ where $\omega_0$ changes from positive to negative (such a value always exists since $\omega_0$ has mean zero). At $t=T$, the trajectory $\Gamma(T,x_0)$ reaches the origin.
\end{theorem}

The transformation $\rho = \sqrt{\eta_x}$ turns the Lagrangian Hunter-Saxton equation into the geodesic equation on the infinite-dimensional sphere, as first noticed by Lenells~\cite{Lenells1}. The fact that solutions in the $\rho$-variable are \emph{global} was used by Lenells~\cite{Lenells2} to obtain global conservative weak solutions. See Figure \ref{figurecap} for the illustration of how $\rho$ goes negative (corresponding to $\eta_x$ approaching zero) and what this looks like in the solar model.

\begin{figure}[h]
\begin{center}
$j=0$: \includegraphics[scale=0.25]{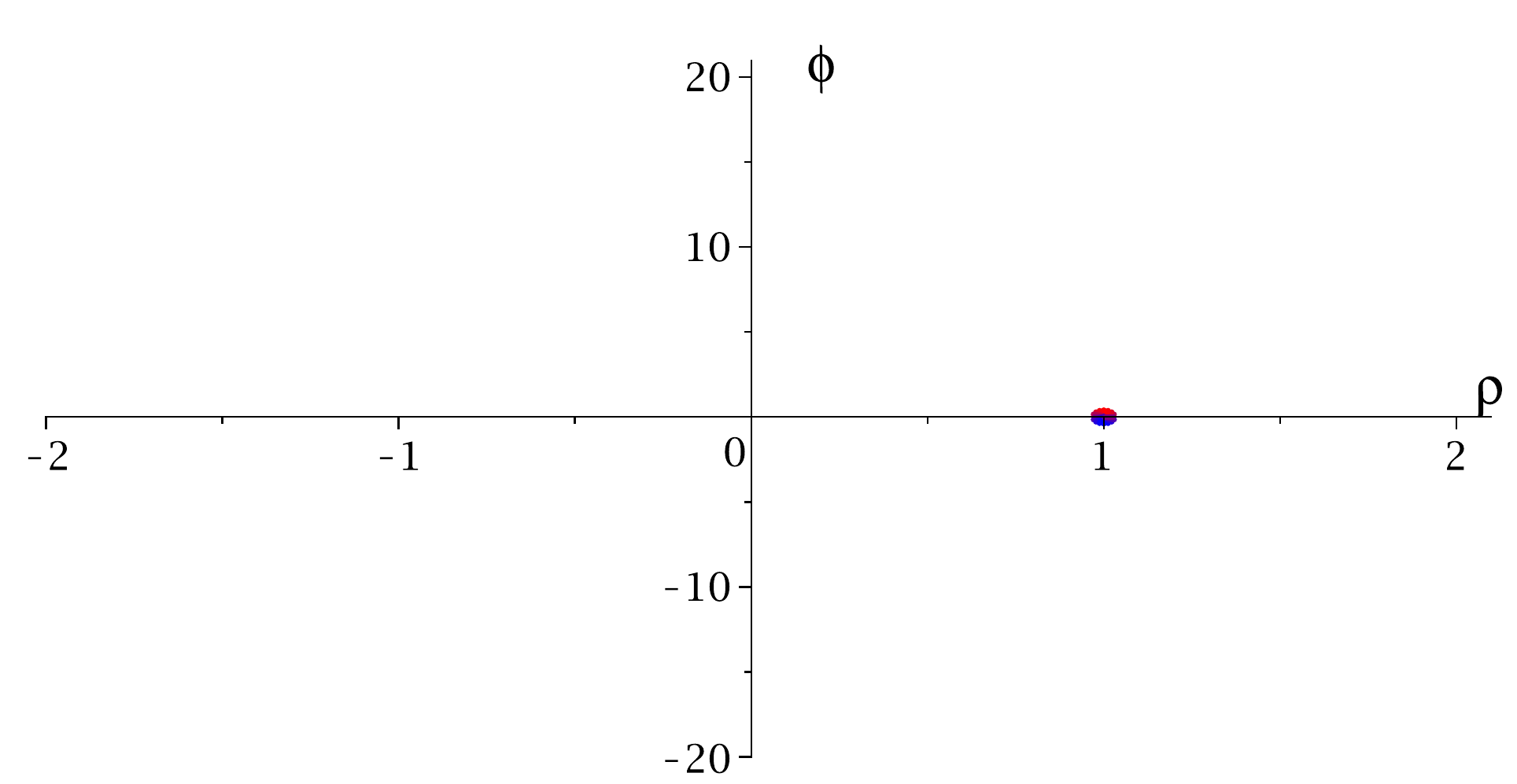} \includegraphics[scale=0.25]{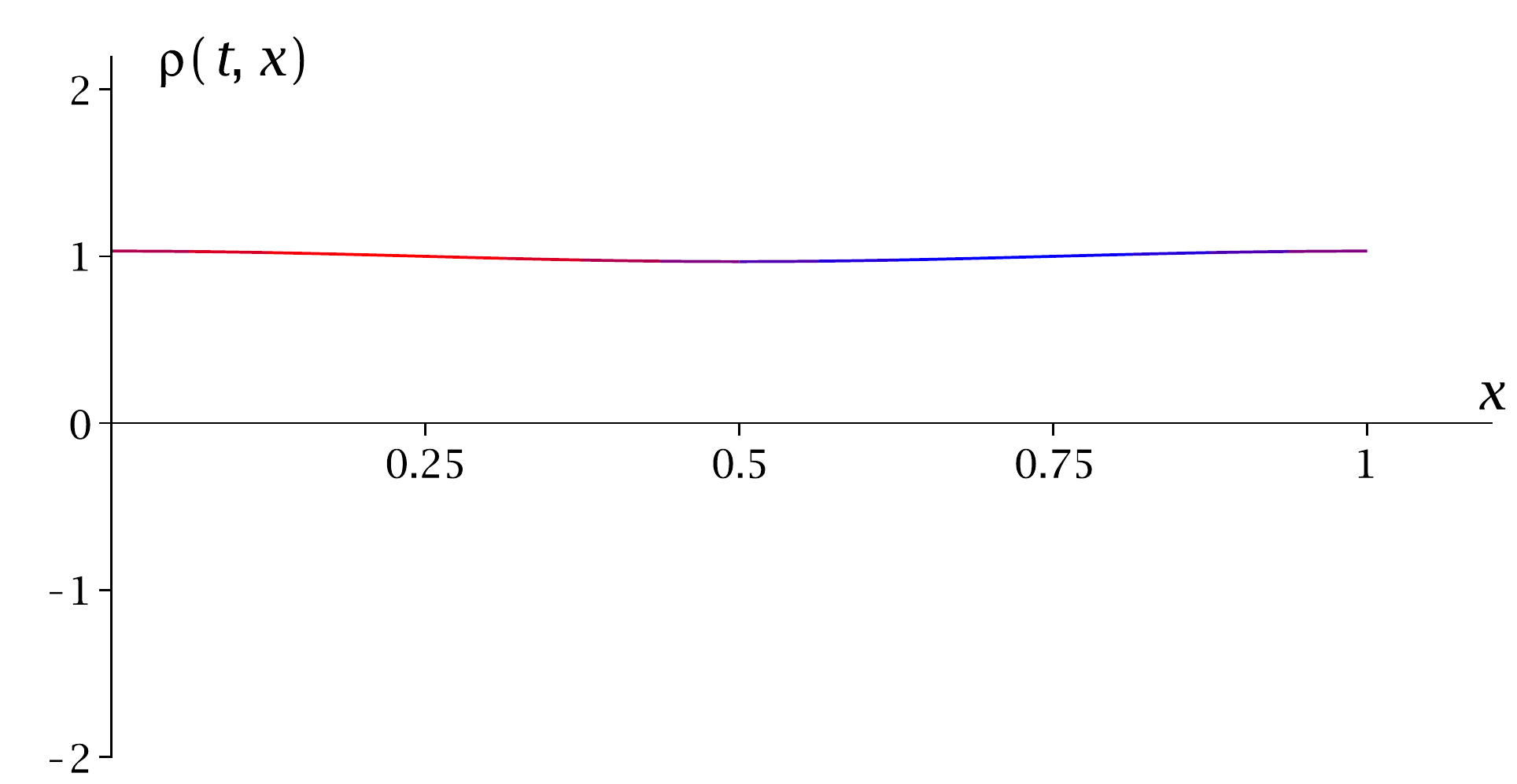} \\
$j=1$: \includegraphics[scale=0.25]{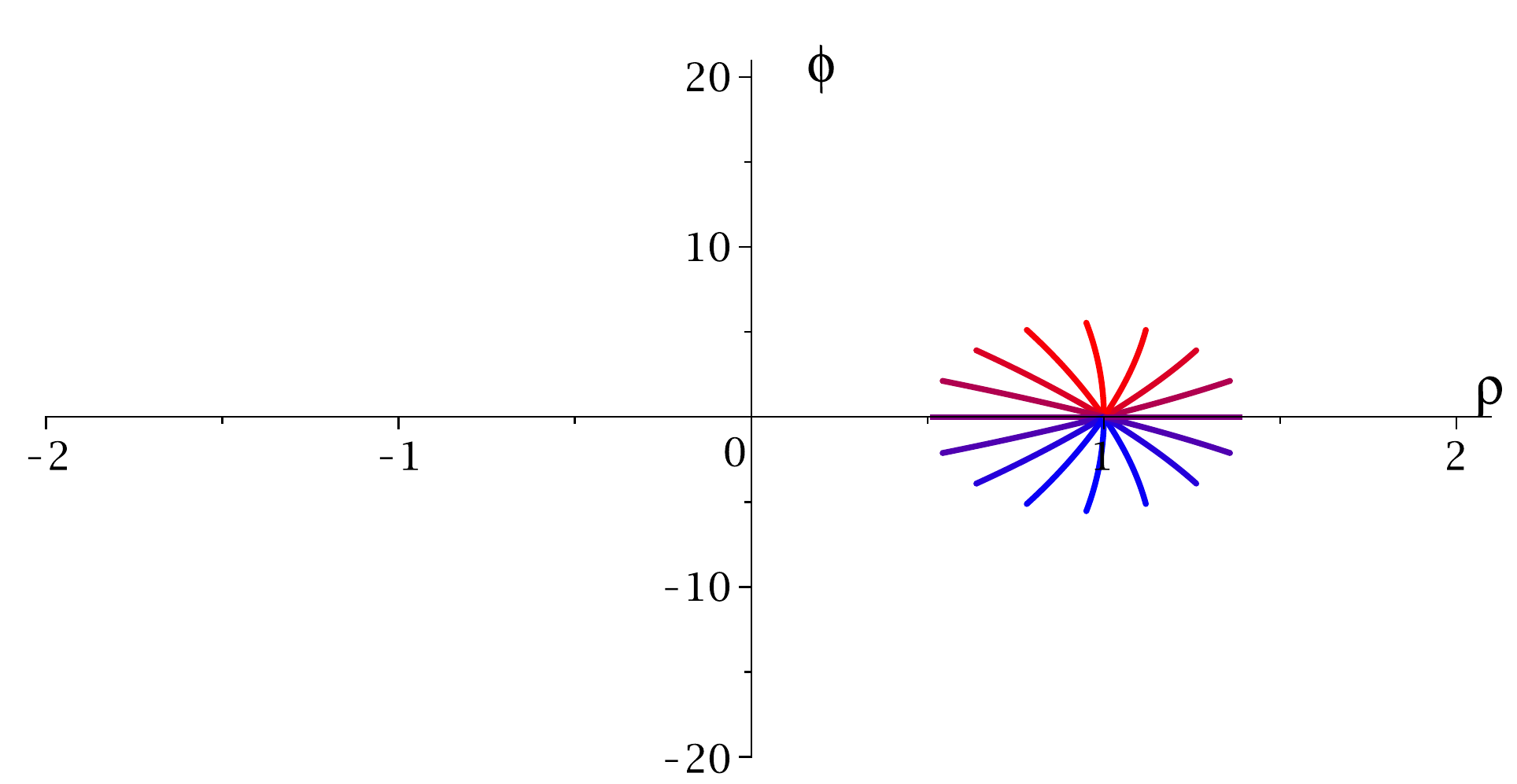} \includegraphics[scale=0.25]{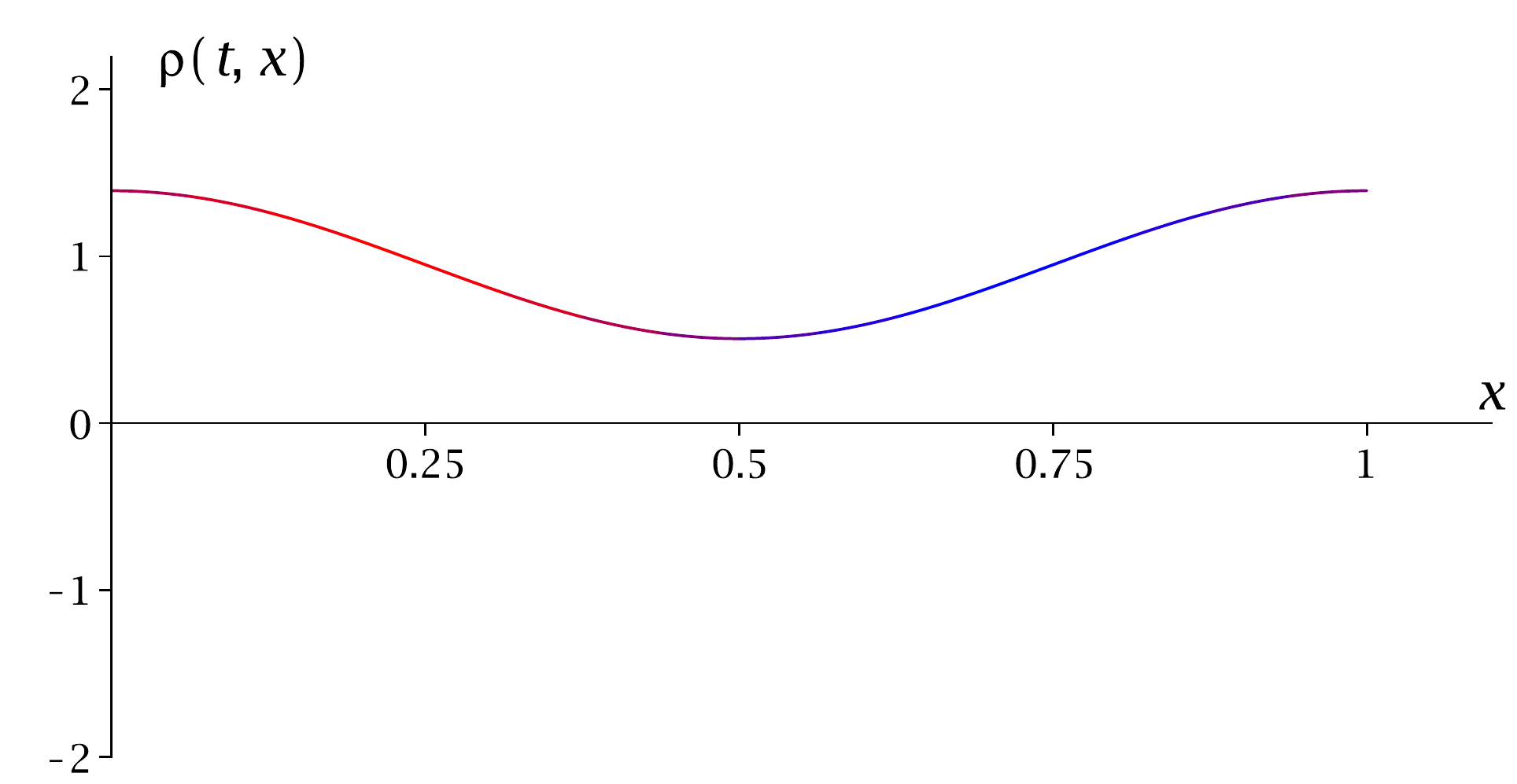} \\
$j=2$: \includegraphics[scale=0.25]{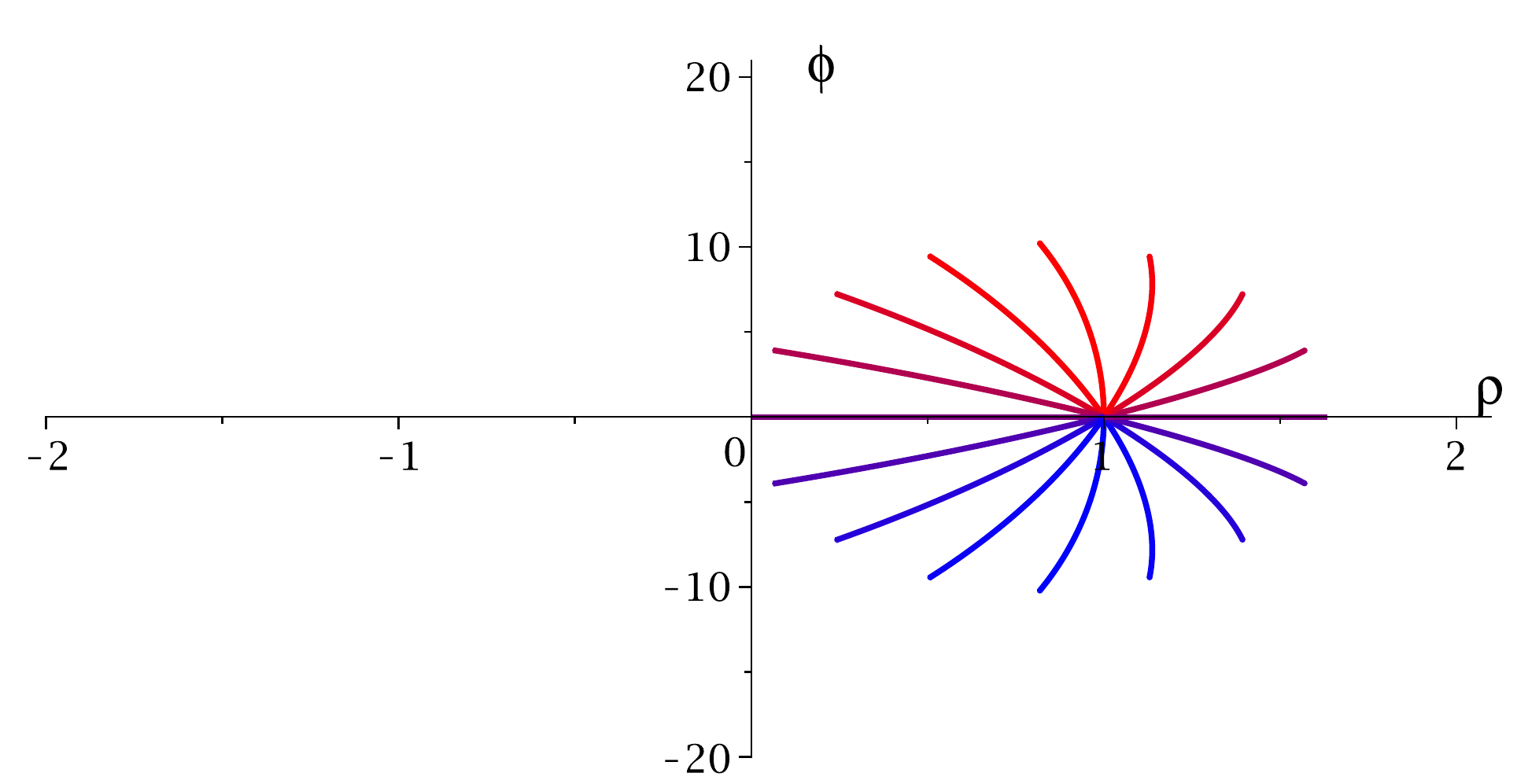} \includegraphics[scale=0.25]{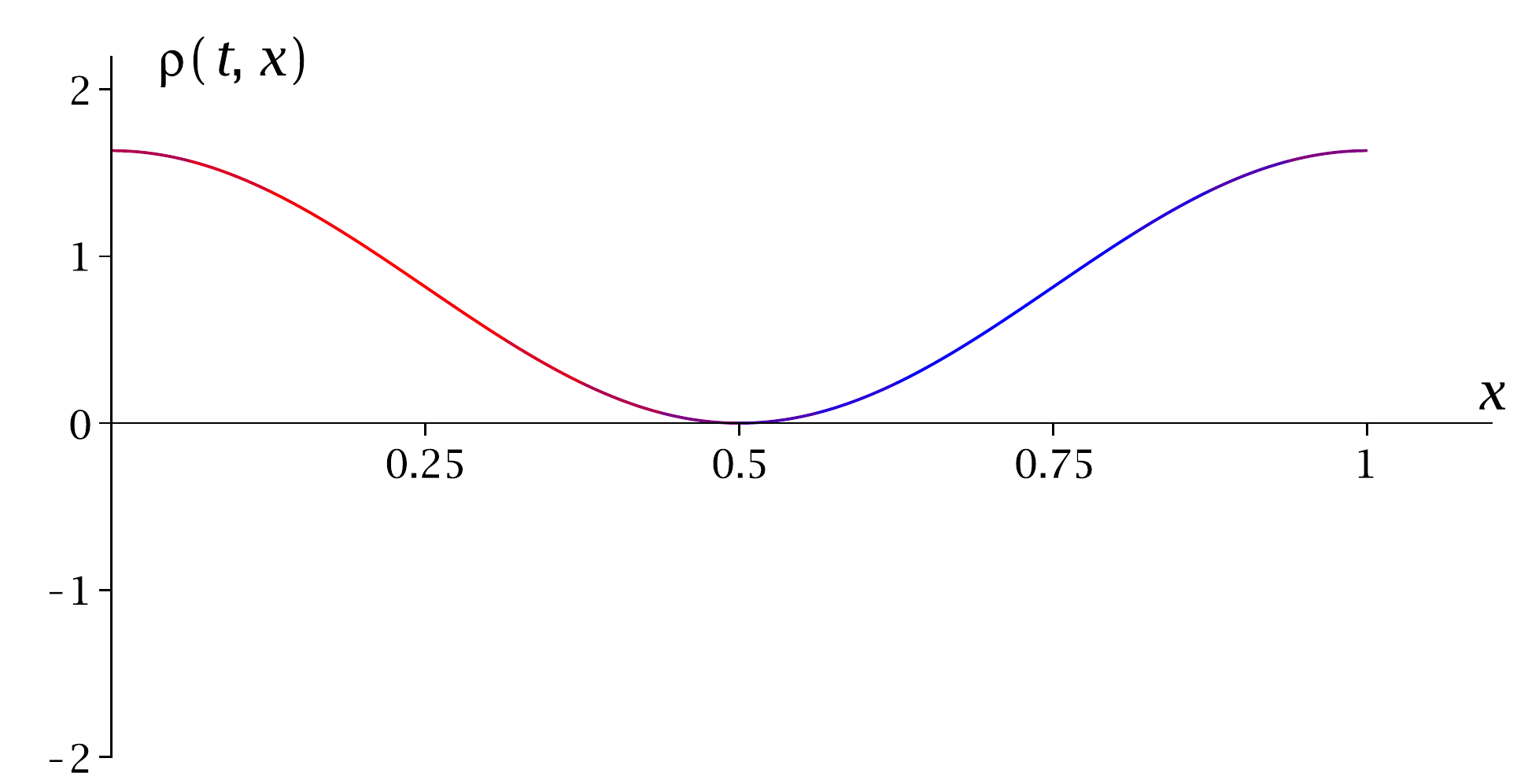} \\
$j=3$: \includegraphics[scale=0.25]{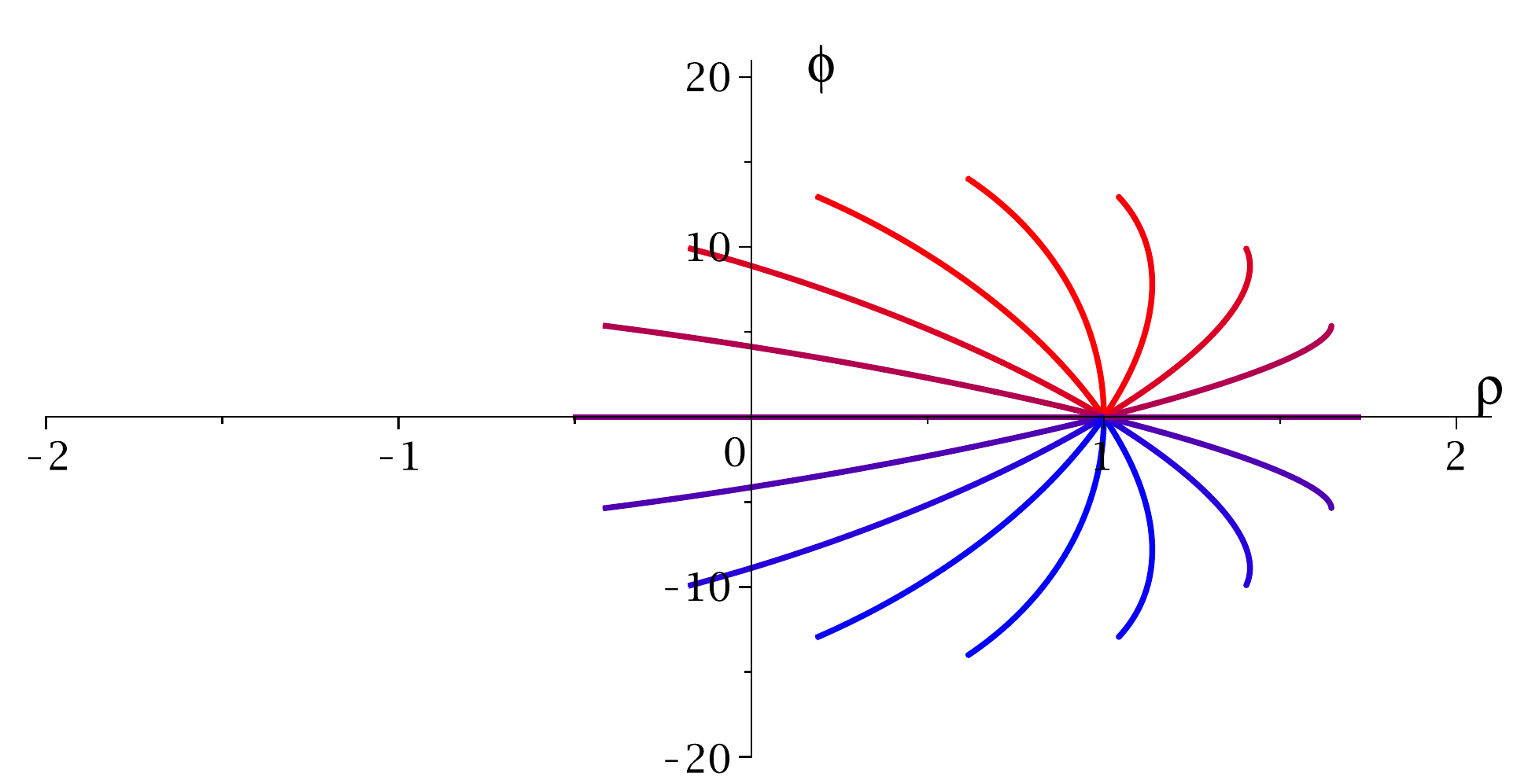} \includegraphics[scale=0.25]{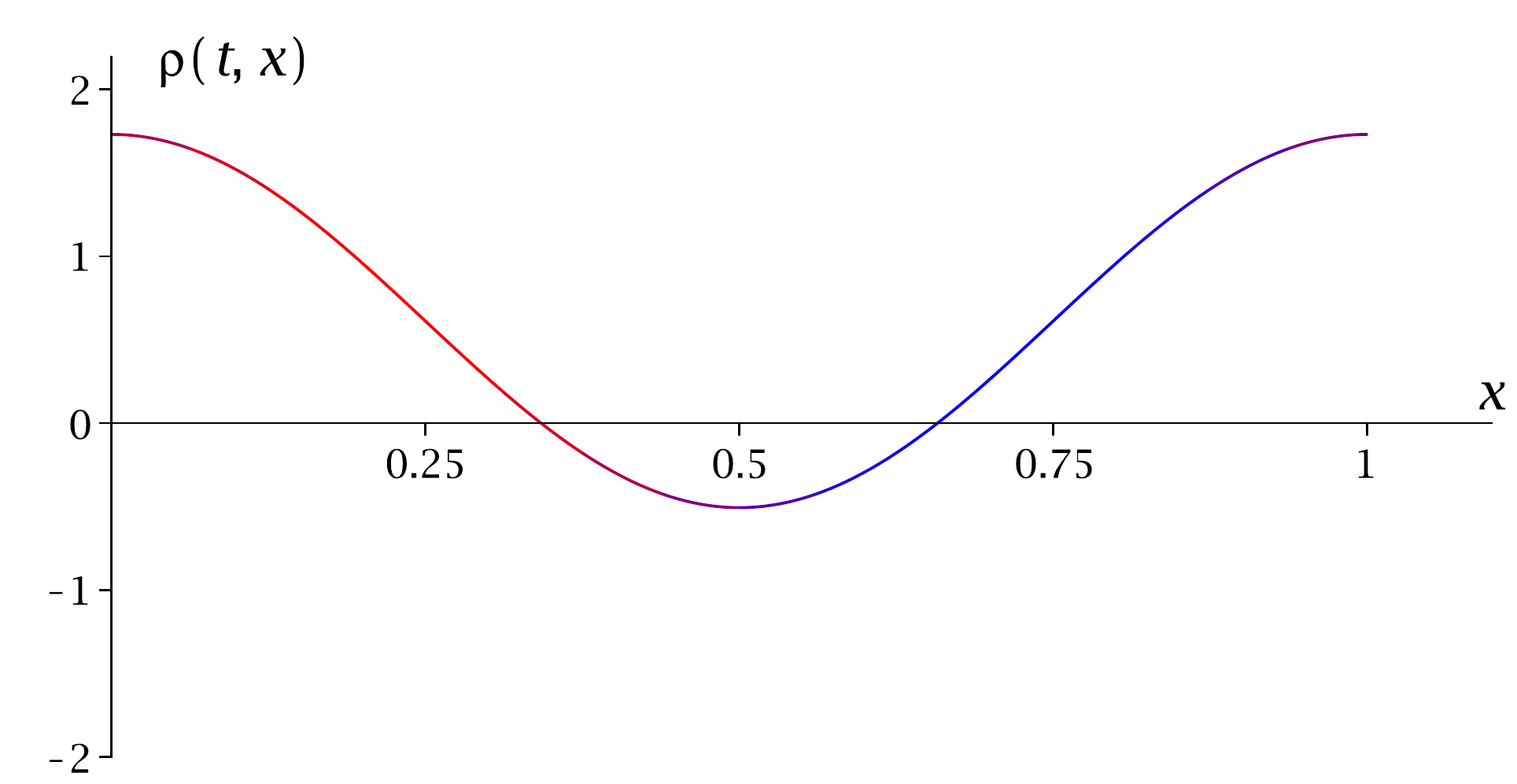} \\
$j=4$: \includegraphics[scale=0.25]{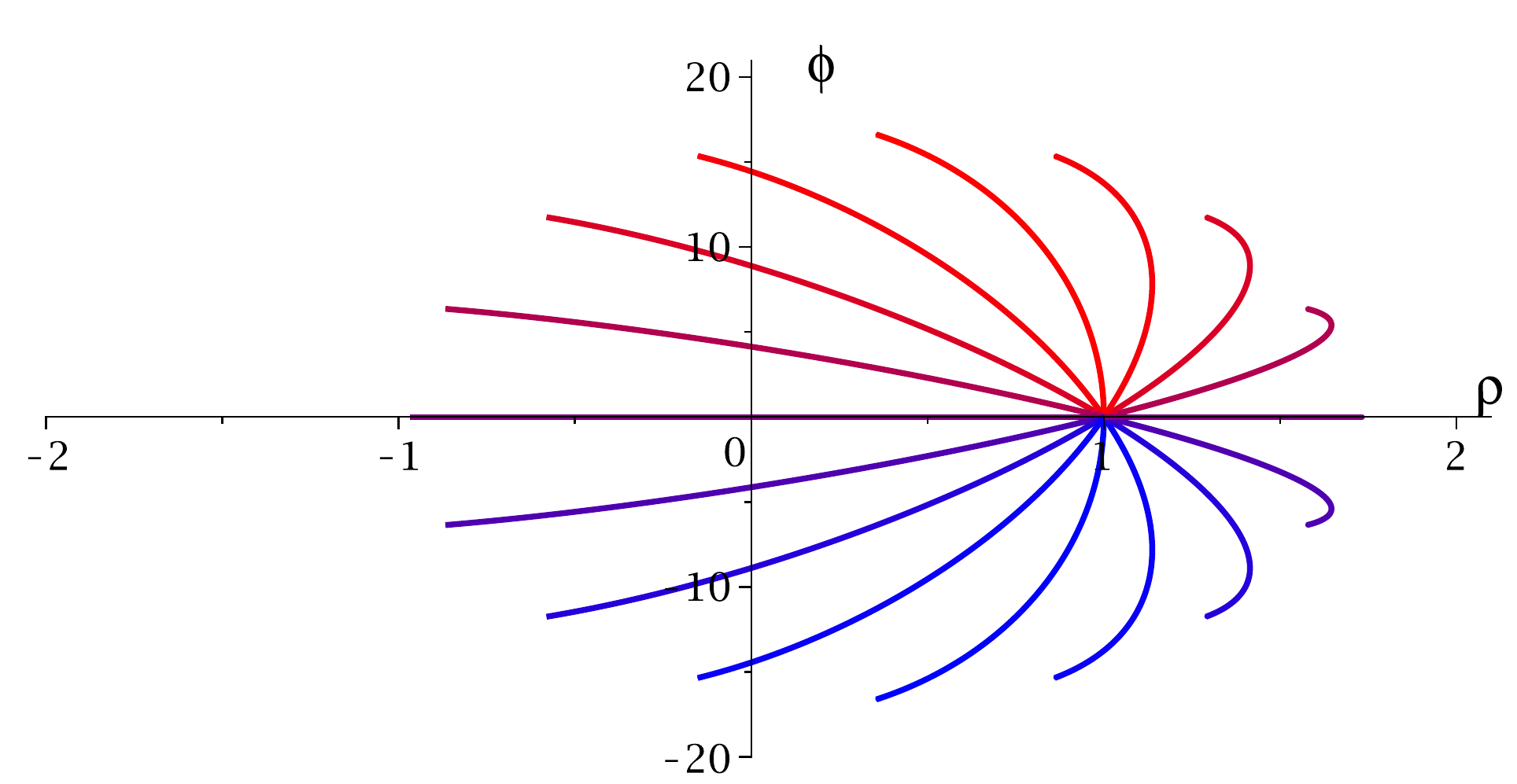} \includegraphics[scale=0.25]{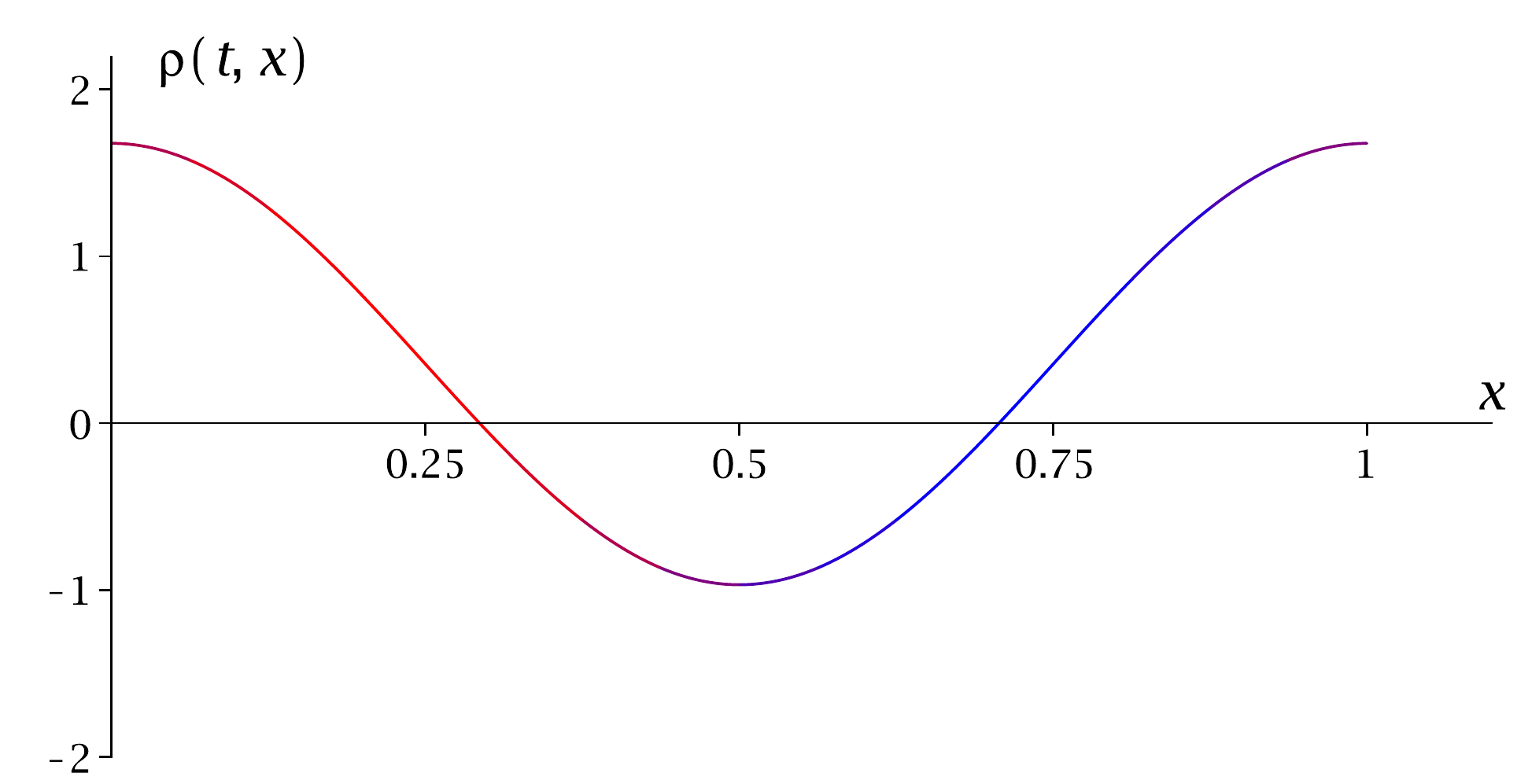} \\ 
\end{center}
\caption{Here we show both the solar model on the left and the solution $\rho(t,x)=\sqrt{\eta_x(t,x)}$ on the right for the Hunter-Saxton equation, with initial condition $u_0(x) = \sin{(2\pi x)}$ for $t_j = 0.01+.1335j$. In the solar model particles emerge from $(1,0)$ with velocity $\langle v_0(x),\omega_0(x)\rangle$ and approach the vertical wall $\rho=0$: first one particle hits the origin coming directly along the horizontal axis at $j=2$, then others follow. On the right $\rho$ and $\rho_x$ have simultaneously reached zero, and the classical solution $u(t,x)$ breaks down. However the solution continues in these variables without noticing. Points colored red have positive angular momentum, while those in blue have negative angular momentum: the first breakdown occurs at the transition.}\label{figurecap}
\end{figure}

For the Wunsch equation we have a similar interpretation. In Lagrangian coordinates it takes the form
\begin{equation}\label{wunschlagrangian}
\eta_{ttx}(t,x) = \frac{\omega_0(x)^2}{\eta_x(t,x)^3} -F\big(t,\eta(t,x)\big)\eta_x(t,x),
\end{equation}
where the function $F$ is a nonlocal expression given by
$$ F(t,x) = -uu_{xx} - H\big(u H(u_{xx})\big);$$
a computation from Bauer-Kolev-Preston~\cite{BKP} shows that $F(t,x)>0$ for all $x$ regardless of $u$. Writing
$$ \Gamma(t,x) = \eta_x(t,x) e^{i\theta(t,x)}$$
where $\theta(t,x) = \omega_0(x) \int_0^t \frac{d\tau}{\eta_x(\tau,x)^2}$ turns the Lagrangian Wunsch equation into
$$ \Gamma_{tt}(t,x) = -F\big(t,\eta(t,x)\big) \Gamma(t,x).$$
We have $\Gamma(0,x) \equiv 1$ and $\Gamma_t(0,x) = u_0'(x) + i\omega_0(x)$. Note that the blowup problem is the same:
$\eta_x(t,x)\to 0$ in finite time if and only if $\Gamma(t,x)$ approaches the origin; on the other hand nonzero
angular momentum prevents this from happening unless the central force $F(t,x)$ is extremely strong. Preston-Washabaugh~\cite{PW}
showed that every solution breaks down in finite time at a point $x_0\in\mathbb{S}^1$ where $u_0'(x_0)<0$ and $\omega_0(x_0)=0$, which always exists.

We conjecture that there is a similar picture for every one-dimensional Euler-Arnold equation. This would entail the following:
\begin{definition}
A \textbf{solar model} for an Euler-Arnold equation on the circle is a function $\Gamma\colon [0,T)\times\mathbb{S}^1 \to \mathbb{R}^2$
such that
\begin{itemize}
\item $\Gamma_{tt}(t,x) = -F(t,x) \Gamma(t,x)$ for some scalar function $F$;
\item The conserved angular momentum $\Gamma_1\dot{\Gamma}_2 - \Gamma_2\dot{\Gamma}_1$ is precisely the conserved Euler-Arnold momentum $\eta_x(t,x)^2 Lu(t,\eta(t,x))$;
\item $\eta(t,x)$ and thus $u(t,x)$ can be recovered from $\Gamma(t,x)$;
\item Breakdown in the form $\eta_x\to 0$ or $u_x\to -\infty$ occurs iff $\Gamma(t,x)$ approaches the origin.
\end{itemize}
\end{definition}

Above we have illustrated solar models for Hunter-Saxton and Wunsch. Solar models should give a way to prove breakdown or global existence in a simpler and more systematic way. Note that the breakdown mechanism is completely understood for the Hunter-Saxton equation (with explicit solutions) and for the Camassa-Holm equation (via an argument of McKean~\cite{McKean} that relies heavily on special properties of that equation).

The problem for now is to devise a solar model for the $\mu$-Hunter Saxton equation, and use it to solve the complete breakdown problem (which remains open). The conjecture is that breakdown occurs iff the momentum $\omega_0(x)$ changes sign, just as happens for the HS, Wunsch, and CH equations.

\end{talk}

\clearpage

\begin{talk}{Stefan Sommer}{Sample Maximum Likelihood Means}{Sommer, Stefan}
The Euclidean expected value $\bar{Y}=\mathbb E[Y]$ of a distribution $Y$ can be generalized to distributions on Riemannian manifolds with, for example, the Fr\'echet mean
\begin{equation}
\bar{Y}_{\mathrm{Fr}}
=
\mathrm{argmin}_x\mathbb E[d(x,Y)^2]
\label{eq:Fr_mean}
\end{equation}
or a maximum likelihood mean
\begin{equation}
\bar{Y}_{\mathrm{ML}}
=
\mathrm{argmax}_x\mathbb E[\log p(Y; x)]
\label{eq:ML_mean}
\end{equation}
where $p(y; x)$ denotes the value at $y$ of the density of a distribution dependent on a parameter $x$. Examples of such distributions include the transition density of a Brownian motion started at $x$ and stopped at some fixed time $T>0$ in which case $p(\cdot; x)$ is the solution to the Riemannian heat equation.

Given i.i.d. samples $y_1,\ldots,y_n\in M$, the sample equivalents of \eqref{eq:Fr_mean} and \eqref{eq:ML_mean}
\begin{align}
&\bar{Y}^n_{\mathrm{Fr}}
=
\mathrm{argmin}_xn^{-1}\sum_{i=1}^n d(x,y_i)^2
\label{eq:Fr_mean_estimator}
\\
&\bar{Y}^n_{\mathrm{ML}}
=
\mathrm{argmax}_xn^{-1}\sum_{i=1}^n[\log p(y_i; x)]
\end{align}
are generally optimized by gradient based iterative optimization. For the Fr\'echet mean, this involves at each iteration computing the distance from the candidate $y$ to the data points, each such computation being in general itself an optimization problem. For the maximum likelihood mean, the likelihood appears in case of an underlying stochastic process starting at $x$ by taking expectation of the sample paths hitting $y_i$
\begin{equation*}
p_T(y_i; x)
=
\frac{\mathbb E[\mathbf 1_{x_T\in dy_i}]}{dy_i}
\ .
\end{equation*}
Computationally, this can be approximated by sampling from the bridge process $x_t|x_T=y_i$ conditioned on hitting $y_i$ at time $t=T$, e.g. \cite{sommer_bridge_2017}. As for the Fr\'echet mean, optimizing over the parameter $x$, the starting point of the process, is a computationally expensive optimization problem.

In the talk, we related these optimization procedures to the central limit theorem (CLT, \cite{bhattacharya_large_2005}) for the Fr\'echet mean estimator \eqref{eq:Fr_mean_estimator} and discussed possibilities for sampling ML means directly from the CLT limiting distribution to avoid the direct optimization procedures.

\end{talk}

\begin{talk}[Andrea Natale]{Fran\c{c}ois-Xavier Vialard}
{Embedding EPDiff equation in incompressible Euler}
{Vialard, Fran\c{c}ois-Xavier}
\section{\texorpdfstring{An $L^2$ embedding of the $H^{\Div}$ right-invariant metric}{An L2 embedding of the H div right-invariant metric}}
Let us consider a closed Riemannian manifold $(M,g)$ and the group of (smooth) diffeomorphisms $\operatorname{Diff}(M)$ endowed with the $H^{\Div}$ right-invariant metric defined as follows:
\begin{definition}
The $H^{\Div}$ right-invariant metric on $\operatorname{Diff}(M)$ is defined at a diffeomorphism $\varphi$ and tangent vector $X_\varphi \in T_\varphi \operatorname{Diff}(M)$ by
\begin{equation}
G(\varphi)(X,X) \eqdef \int_M g(X_\varphi \circ \varphi^{-1},X_\varphi \circ \varphi^{-1}) \ud\!\vol + \frac 14 \int_M | \Div(X_\varphi \circ \varphi^{-1}) |^2 \ud\!\vol\,. 
\end{equation}
where $\Div$ denotes the divergence operator and $\vol$ is the Riemannian volume form of $(M,g)$.
\end{definition}
The embedding of the diffeomorphism group with this $H^{\Div}$ metric is based on the following idea, write redundantly $\varphi$ as $(\varphi,\Jac(\varphi))$ as an element $(\varphi,\lambda)$, then we rewrite the metric in terms of these two quantities. To do so, we observe that the tangent vector to $\lambda = \Jac(\varphi)$ denoted $X_\lambda$ is given by $ X_\lambda = \Div(X_\varphi \circ \varphi^{-1}) \lambda$. Now, the metric is obtained by a change of variable on the $L^2$ norm of the first term of the metric and the second term is obtained by direct rewriting:
\begin{equation}\label{EqMetric}
G(\varphi)(X,X) = \int_M g(X ,X) \lambda \ud\!\vol + \frac 14 \int_M \frac{X_\lambda^2}{ \lambda} \ud\!\vol\,. 
\end{equation}
\begin{definition}\label{DefCone}
The Riemannian cone $\mathcal{C}(M)$ is the product manifold $M \times (0,+\infty)$ endowed with the metric $\lambda g + \frac 14 \frac{\ud \lambda^2}{\lambda}$.
\end{definition}
We call it a cone metric because with the change of variable $r^2 = \lambda$ we get $r^2g + \ud r^2$ which is the usual cone metric in Riemannian geometry. In particular, if $M = S_1$, this metric is the Euclidean metric on $\R^2 \setminus \{ 0 \}$. 
Note that the metric defined in \eqref{EqMetric} is a metric on the space of maps from $M$ into $M \times (0,+\infty)$. It is actually the $L^2$ metric on $L^2(M,\mathcal{C}(M))$ where $M$ is endowed the Riemannian volume measure and $\mathcal{C}(M)$ endowed with the cone metric.
Therefore, we have, as a direct consequence of the definitions
\begin{theorem}
The map 
\begin{align*}
\operatorname{Inj}:&\operatorname{Diff}(M) \to L^2(M,\mathcal{C}(M))\\
&\varphi \mapsto  (\varphi,\Jac(\varphi))\,.
\end{align*}
is an isometric embedding between the right-invariant metric $H^{\Div}$ and the $L^2_{\ud\!\vol}(M,\mathcal{C}(M))$ (non right-invariant) metric.
\end{theorem}
The obvious application is the following:
\begin{theorem}[Michor and Mumford 2005]
The distance on $\operatorname{Diff}(M)$ with the right-invariant $H^{\Div}$ metric is non-degenerate (the infimum of the length of the paths connecting two given distinct diffeomorphisms is not zero).
\end{theorem}
\begin{proof}
The length of any path is bounded below by the length of the geodesic in the space $L^2(M,\mathcal{C}(M))$ which is not null if the two diffeomorphisms are different.
\end{proof}

\section{Embeddings in incompressible Euler}

Using an isomorphism between the semi-direct product of groups $\operatorname{Diff}(M) \ltimes C(M,\R_{>0})$ and the automorphism group of the cone $\mathcal{C}(M)$, it is possible to prove that the geodesic flow for the $H^{\Div}$ right-invariant metric can be embedded in the space of solutions to the incompressible Euler equation via the following map, close to a Madelung transformation (up to a square root change of variables on the Jacobian):
\begin{equation}
[x \mapsto \varphi(x)] \longmapsto [(x,r) \mapsto (\varphi(x),r\Jac(\varphi)(x)]
\end{equation}
This transformation maps diffeomorphisms of $M$ to automorphisms of $\mathcal{C}(M)$ which are volume preserving for the density $\frac{1}{r^2}\ud r \ud \!\vol_M$. From this remark, it is possible to prove that the $H^{\Div}$ geodesic flow can be embedded in the incompressible Euler flow on the cone for this particular density. One can also get rid of this density by adding a dummy dimension as done in \cite{Tao}. As a result, one can embed the Camassa--Holm equation in incompressible Euler of a $3$ dimensional curved manifold, see \cite{Natale1} for more details. In this paper, we also show that it is also possible to embed the $CH2$ equation, which is a generalization of the Camassa--Holm equation and can be presented as a geodesic flow for a right-invariant metric on a semi-direct product of groups.
\par
Therefore, we are led to the following question,
\begin{question}
For which right-invariant metrics on $\operatorname{Diff}(M)$ such an embedding into the incompressible Euler equation of a Riemannian manifold $N$ is possible ?
\end{question}
One could start with two cases of interest: (1) the $H^1$ metric on the group of diffeomorphisms in dimension greater than $1$ and (2) A higher-order Sobolev metric on $S_1$.

\end{talk}

\end{report}

\end{document}